\documentclass{siamart190516}

\usepackage[utf8]{inputenc}
\usepackage{graphicx}
\usepackage{subcaption}
\usepackage{amsfonts}
\usepackage{epstopdf}
\usepackage{amsmath,accents}
\usepackage{amssymb}
\usepackage{soul}
\usepackage{tikz}
\usetikzlibrary{arrows.meta}

\ifpdf
  \DeclareGraphicsExtensions{.eps,.pdf,.png,.jpg}
\else
  \DeclareGraphicsExtensions{.eps}
\fi

\newsiamremark{remark}{Remark}
\newsiamremark{hypothesis}{Hypothesis}
\crefname{hypothesis}{Hypothesis}{Hypotheses}
\newsiamthm{claim}{Claim}

\headers{A mean field game of innovation}{M. Barker, P. Degond, R. Martin, M. Mu\^uls}

\title{A mean field game model of firm--level innovation}

\author{Matt Barker\thanks{Science and Solutions for a Changing Planet DTP, and the Department of Mathematics, Imperial College London, London, UK 
  (\email{m.barker17@imperial.ac.uk}). This work was supported by the Natural Environment Research Council [grant number NE/L002515/1]}
\and Pierre Degond\thanks{Department of Mathematics, Imperial College London, London, UK 
  (\email{p.degond@imperial.ac.uk}). PD acknowledges support by the Engineering and Physical Sciences Research Council (EPSRC) under grants no. EP/M006883/1 and by the Royal Society and the Wolfson Foundation through a Royal Society Wolfson Research Merit Award no. WM130048. PD's affiliation from 01/01/2021 is "Institut de Mathématiques de Toulouse, CNRS \& Université Paul Sabatier, 31062 TOULOUSE, France. MB is supported by the Natural Environment Research Council}
\and Ralf Martin\thanks{Imperial Business School, Imperial College London, London, UK
  (\email{r.martin@imperial.ac.uk}, \email{m.muuls@imperial.ac.uk}).}
\and Mirabelle Mu\^uls}

\usepackage{amsopn}

\ifpdf
\hypersetup{
  pdftitle={An Example Article},
  pdfauthor={D. Doe, P. T. Frank, and J. E. Smith}
}
\fi

\DeclareMathOperator{\E}{\mathbb{E}}
\DeclareMathOperator{\R}{\mathbb{R}}
\DeclareMathOperator{\N}{\mathbb{N}}

\begin{document}

\maketitle
\begin{abstract}
    Knowledge spillovers occur when a firm researches a new technology and that technology is adapted or adopted by another firm, resulting in a social value of the technology that is larger than the initially predicted private value. As a result, firms systematically under--invest in research compared with the socially optimal investment strategy. Understanding the level of under--investment, as well as policies to correct it, is an area of active economic research. In this paper, we develop a new model of spillovers, taking inspiration from the available microeconomic data. We prove existence and uniqueness of solutions to the model, and we conduct some initial simulations to understand how indirect spillovers contribute to the productivity of a sector.
\end{abstract}



\section{Introduction}
When a business invests in research and development (R\&D), such strategy only takes into account how a potential innovation may increase the investing company's private value. However, other businesses may utilise innovations made by the original investing company to increase their own profits. This is known in economic literature as the knowledge spillover effect. By only considering its private return, businesses systematically undervalue their own innovations and hence under--invest in R\&D, compared with the socially optimal investment level. To counteract the under--investment, governments introduce R\&D subsidy policies for certain sectors of the economy. In order to effectively allocate such subsidies, it is therefore important to understand the extent of under--investment and how it varies between sectors. 

To understand the spillover effect we develop a mean field game (MFG) model of firms distributed heterogeneously between sectors and according to their productivity level, taking into account their microscopic behaviour. From a microeconomic perspective, the size of knowledge spillovers can be inferred from the network of patent citations~\cite{Fung2005}. When an industrial technology is developed, it often gets patented. As part of the patent any previous technology that has been used must be cited. This results in a network of patent citations, where each citation can be used as a proxy for a spillover from one technology to another, so spillover sizes can be evaluated~\cite{Dechezlepretre2013}. In the model we develop, sectors are connected by a graph that is informed by and can be calibrated to the microeconomic network of patent citation data.

A first model of knowledge spillovers, by Cohen and Levinthal~\cite{Levinthal1989}, considered the stock of knowledge of a firm to depend on the amount of investment in R\&D of that firm and the total amount of investment by all other firms, through a mean field--type interaction. Only an initial analysis of the model was conducted in~\cite{Levinthal1989}. A later model, acknowledged in Section 13.2 of~\cite{Acemoglu2009}, started from a macroscopic perspective, hence only the aggregate knowledge of the entire economy was considered and spillovers were assumed to increase the aggregate uniformly. This does not explain how spillovers heterogeneously affect firms. Similar models have also been used to study entrepreneurship and intellectual property rights, such as in~\cite{Acs2012}. There has been particularly extensive research of knowledge spillovers in cross--country models. In such models, a country's own output is aggregated and the knowledge level increases at a rate that depends on the leading country's knowledge level. However, simplifying assumptions are made that may affect their accuracy, such as in~\cite{Eeckhout2002}, where interactions take place in a discrete time setting, or~\cite{Aghion2005,Howitt2002} where the interactions between firms was described only through the evolution of aggregate quantities. In this paper, we use an MFG model to both increase the complexity of the description of firms, and to link firm--level evolution directly to microeconomic data for spillover sizes. There have been several other papers focussing on MFG--type models of knowledge spillovers, see~\cite{Burger2016,Lucas2014}. The Boltzmann model studied in both papers doeas not consider how innovation among firms evolves, nor did it incorporate the microeconomic data related to patent citations in its formulation. As a result, the model studied in this paper can give greater insight into firm--level dynamics.

In this paper, we analyse a stationary MFG model describing the spillover effect. The MFG model describes the long--term behaviour of firms with full anticipation of the future. MFGs were described mathematically by Lasry and Lions~\cite{Lasry2006,Lasry2007}, and simultaneously by Huang, Caines and Malham\'e~\cite{Huang2006} and they build on the work of Aumann and related authors on anonymous games~\cite{Aumann1964,Schmeidler1973}. The novelties of the system we develop are, first, that the distribution dependence enters into the drift term rather than in the cost functional and, second, that we are considering more than one population of agents. Therefore our MFG model can be classed as a multi--population MFG with a non--separable Hamiltonian. There has been some work in both multi--population MFGs (see~\cite{Cirant2015}) and MFG models with non--separable Hamiltonians (see~\cite{Ambrose2018,Ferreira2018}). However, we are aware of no literature for models that display both characteristics, so although our model is one dimensional, its interest reaches beyond this setting. As a result of the novelty of our model, the techniques we use to prove existence and uniqueness are also novel. However, they rely heavily on the ability to write a stationary Fokker--Planck equation in the form of an exponential. This characteristic has previously been used in~\cite{Barker2019} to prove existence and uniqueness in MFG and BRS models in a slightly different framework.

The paper is organised as follows. In Section 2, we develop the spillover model by describing firm behaviour at a microscopic level and formally deriving the mean field limit. In section 3, we describe the MFG problem and prove existence of solutions to it. We also show uniqueness of such solutions holds, provided the coupling strength between sectors is small enough. In Section 4, we provide some deeper insights into the effects of the modelling parameters, through numerical simulations. The first simulations show how parameters describing effects unrelated to spillovers (for example the discount factor, the noise level and the labour efficiency) change the MFG model. Our second group of simulations demonstrate the effect of the spillover network on the model. The spillover network is a sector--level network that aggregates the patent citation network. We show that the effect of a spillover on any sector is a result of all paths to that sector in the associated network, and not just the immediate connections between sectors, which is contrary to the current economic state of the art. Finally, in Section 5, we briefly discuss future research prospects for the model, including how we intend to apply the model to economic questions relevant to R\&D subsidy policy.

\section{Model development} \label{sec:mfg}
\subsection{The microscopic model} \label{sec:mfg-micro-model}
\subsubsection*{Firms}
Assume there are $L$ sectors within the economy, and in sector $\ell$ there are $N_{\ell}$ firms. We assume firm $i$ in group $\ell$ has $s^{\ell,\ell'}_{i,j}$ links with firm $j$ in group $\ell'$, where $s^{\ell,\ell'}_{i,j}$ is a random variable, taking a value $s \in \N$ with probability $\frac{1}{N_{\ell'}} p(\ell,\ell',s)$. The $i^{th}$ firm in sector $\ell$ has a productivity level $Z_{\ell,i} \in \Omega = (0,\bar{z})$, which increases as a result of employing labour $h_{\ell,i}$ or due to knowledge spillovers from firms that they are linked with. The productivity dynamics are also affected by noise with strength $\sigma \in (0,\infty)$. As a result, $Z_{\ell,i}$ evolves according to the following SDE
\begin{subequations} \label{eq:mfg-micro-dyn}
    \begin{align}
        & dZ_{\ell,i}(t) = \left( \left( h_{\ell,i}(t) \right)^{\gamma} + \frac{1}{N} \sum_{\ell' = 1}^L \sum_{j = 1}^{N_{\ell'}} s^{\ell,\ell'}_{i,j} Z_{\ell',j}(t) \right) dt + \sigma dB_{\ell,i}(t) \\
        & \mathcal{L}\left(Z_{\ell,i}(0)\right) = m_{\ell}^0 \, , \label{eq:mfg-micro-dyn-init}
    \end{align}
\end{subequations}
where $N = \sum_{\ell = 1}^L N_{\ell}$, $B_{\ell,i}$ is an independent Brownian motion with reflection at boundaries 0 and $\bar{z}$, and $\gamma \in (0,1)$ represents the inefficiency in converting one unit of labour to one unit of knowledge. In the initial condition~\eqref{eq:mfg-micro-dyn-init}, $\mathcal{L}\left(Z_{\ell,i}(0)\right)$ denotes the law of the random variable $Z_{\ell,i}(0)$ and $m_{\ell}^0$ is an initial distribution, which may be different for each sector. We assume firms produce a quantity of differentiated good at a rate $q_{\ell,i}$ according to the production function
\begin{equation} \label{eq:productivity-to-goods}
    q_{\ell,i} = Z_{\ell,i} \, .
\end{equation}
Each firm sells their product at a market--determined price $r_{\ell,i}$ and maximises their profit subject to the other firms' decisions. Each agent's profit functional is given by
\begin{equation} \label{eq:micro-profit}
   J_{\ell,i}(h) = \E \left[ \int_0^\infty \left( r_{\ell,i}(t) q_{\ell,i}(t) - w h_{\ell,i}(t) \right) e^{- \rho t}~dt \right] \, ,
\end{equation}
where $h = \left( \left( h_{\ell,i} \right)_{i = 1}^{N_{\ell}} \right)_{\ell = 1}^L$. The wage, $w$, and the discount rate, $\rho$, are given constants.

\subsubsection*{Consumers}
Assume there is a representative consumer with preferences given by $Q = \frac{1}{N} \sum_{\ell = 1}^L\sum_{i = 1}^{N_{\ell}} q_{\ell,i}^{\alpha}$, the Dixit--Stiglitz constant elasticity of substitution (CES) form, and with average income $Y$. The value $\alpha \in (0,1)$ is related to the elasticity of substitution. The demand for each variety can be found by maximising $Q$ under the budget constraint that average expenditure is equal to average income, i.e. $\frac{1}{N}{\sum_{\ell = 1}^L\sum_{i = 1}^{N_{\ell}} r_{\ell,i} q_{\ell,i} = Y}$. This gives
\begin{equation} \label{eq:consumer-pricing}
    q_{\ell,i} = B r_{\ell,i}^{\frac{1}{\alpha - 1}} \, , \quad B = Y R^{\frac{\alpha}{1 - \alpha}} \, , \quad R = \left( \frac{1}{N} \sum_{\ell = 1}^L\sum_{i = 1}^{N_{\ell}} r_{\ell,i}^{\frac{\alpha}{\alpha - 1}} \right)^{\frac{\alpha - 1}{\alpha}} \, .
\end{equation}
For the purposes of the firm--level optimisation problem, we assume $R$ is fixed, in that it can't be changed by any individual firm --- this becomes true as $N_{\ell} \to \infty$ for each $\ell$. For the mathematical analysis we assume $B$ to be a fixed constant, which simplifies matters and enhances the model's relevant features. Later in the numerical simulations $B$ will be determined as the solution to a fixed point problem, which endogenises the price formation through the interaction between firms and consumers.

\subsubsection*{Firm profits revisited}
Now, the profit functional~\eqref{eq:micro-profit} can be rewritten as $J_{\ell,i}(h) = \E \left[ \int_0^\infty \left( \frac{\left(Z_{\ell,i}(t)\right)^{\alpha}}{B^{\alpha - 1}} - w h_{\ell,i}(t) \right) e^{- \rho t}~dt \right]$, using the consumer behaviour~\eqref{eq:consumer-pricing} and the production function~\eqref{eq:productivity-to-goods}.

\subsection{Mean field limit}
When there are large numbers of firms in each sector, the microscopic model developed in Section~\ref{sec:mfg-micro-model} can become intractable. Instead, we assume the number of firms in each sector, $N_{\ell}$, goes to infinity while $\frac{N_{\ell}}{N} \to A_{\ell}$ for some $A_{\ell} \in (0,1)$, which represents the proportion of firms in sector ${\ell}$. In order to derive the limiting mean field model, we first define the empirical distributions for each sector ${\ell} = 1, \ldots, L$ by $m^{N_{\ell}}_{\ell} = \frac{1}{N_{\ell}} \sum_{i = 1}^{N_{\ell}} \delta_{Z_{\ell,i}}$, where $\delta_{Z_{\ell,i}}$ is a Dirac delta at the point $Z_{\ell,i}$. We can then rewrite the dynamics~\eqref{eq:mfg-micro-dyn} using $m^{N_{\ell}}_{\ell}$ as
\[ \begin{aligned}
    & dZ_{\ell,i}(t) = \left( \left( h_{\ell,i}(t) \right)^{\gamma} + \sum_{\ell' = 1}^L \frac{(N_{\ell'})^2}{N} \int_{\Omega} z'~dm^{N_{\ell'}}_{\ell'}(z',t) \frac{1}{N_{\ell'}} \sum_{j = 1}^{N_{\ell'}} s^{\ell,\ell'}_{i,j} \right) dt + \sigma dB_{\ell,i}(t) \\
    & \mathcal{L}\left(Z_{\ell,i}(0)\right) = m_{\ell}^0 \, ,
\end{aligned} \]
Assuming $ m^{N_{\ell}}_{\ell}$ has a limit, $m_{\ell}$, as $N_{\ell} \to \infty$ then, in the limiting model, a representative firm in sector ${\ell}$ evolves according to the SDE
\begin{subequations} \label{eq:mean-field-dynamics}
    \begin{align}
    & dZ_{\ell}^{h,m}(t) = \left( \left( h_{\ell}(t) \right)^{\gamma} + \sum_{\ell' = 1}^L A_{\ell'} p(\ell,\ell') \int_{\Omega} z'~dm_{\ell'}(z',t) \right) dt + \sigma dB_{\ell}(t) \\
    & \mathcal{L}\left(Z_{\ell}^{h,m}(0)\right) = m_{\ell}^0 \, ,
    \end{align}
\end{subequations}
where, by the law of large numbers, $p(\ell,\ell') = \sum_{s = 0}^{\infty} p(\ell,\ell',s)$.  The corresponding profit functional is
\begin{equation} \label{eq:mean-field-profit}
    J_{\ell}(h;m) = \E \left[ \int_0^\infty \left( \frac{\left(Z_{\ell}^{h,m}(t) \right)^{\alpha}}{B^{\alpha - 1}} - w h_{\ell}(t) \right) e^{- \rho t}~dt \right] \, .
\end{equation}
If all firms act in the same way as the representative firm, then the distribution of firms with respect to productivity level is given by a system of $L$ Fokker--Planck equations
\begin{subequations} \label{eq:mfg-FP-gen}
    \begin{align}
        & \partial_t m_{\ell} = - \partial_z \left[ \left( \left( h_{\ell} \right)^{\gamma} + \sum_{\ell' = 1}^L A_{\ell'} p(\ell,\ell') \int_{\Omega} z'~dm_{\ell'}(z',t) \right) m_{\ell} \right] + \frac{\sigma^2}{2} \partial_{zz}^2 m_{\ell} \\
        & \left. - \left( \left( h_{\ell} \right)^{\gamma} + \sum_{\ell' = 1}^L A_{\ell'} p(\ell,\ell') \int_{\Omega} z'~dm_{\ell'}(z',t) \right) m_{\ell} + \frac{\sigma^2}{2} \partial_z m_{\ell} \right|_{z = 0, \bar{z}} = 0 \\
        & m_{\ell}(z,0) = m_{\ell}^0(z) \, .
    \end{align}
\end{subequations}

\section{The MFG model} \label{sec:mfg-mean-field-form}
\subsection{Problem formulation}
The MFG problem is related to the search for Nash equilibria in the optimisation of the profit functional~\eqref{eq:mean-field-profit} while agents evolve according to the dynamics~\eqref{eq:mean-field-dynamics}.

\begin{definition}
    The MFG problem is to find a pair $(h^*,m^*)$, where $h^* = \left( h^*_{\ell} \right)_{\ell = 1}^L$ is a sequence of controls and $m^* = \left( m^*_{\ell} \right)_{\ell = 1}^L$ is a sequence of probability distributions on $\bar{\Omega}$, such that for any other sequence of controls $h$ and every $\ell$
    \begin{subequations}
        \begin{align}
            J_{\ell}(h_{\ell}^*,m^*) \geq J_{\ell}(h_{\ell},m^*) \label{eq:mfg-opt-prob} \\
            \text{and } m_{\ell}^* = \mathcal{L}\left( Z_{\ell}^{h^*,m^*} \right) \, . \label{eq:mfg-consist-prob}
        \end{align}
    \end{subequations}
    Such a distribution is called an MFG equilibrium.
\end{definition}

To find an MFG equilibrium we first describe the Hamilton--Jacobi--Bellman (HJB) PDE related to the optimisation part of the problem~\eqref{eq:mfg-opt-prob}. Then we couple the HJB PDE to the Fokker--Planck PDE~\eqref{eq:mfg-FP-gen} to solve the consistency part~\eqref{eq:mfg-consist-prob}. We start by defining $L$ Hamiltonians $H_{\ell}:\Omega \times \left( H^1 (\Omega) \right)^L \times \R \to \R$, for $\ell = 1, \ldots L$ as 
\begin{equation} \label{eq:mfg-ham}
    \begin{aligned}
        & H_{\ell}(z,m,\lambda) = \sup_{h \geq 0} \left( h^{\gamma} + \sum_{\ell'=1}^L A_{\ell'} p(\ell,\ell') \int_{\Omega} z' m_{\ell'}(z')~dz'\right) \lambda + \frac{z^{\alpha}}{B^{\alpha - 1}} - w h \\
        & = (1 - \gamma) \left( \frac{\gamma}{w} \right)^{\frac{\gamma}{1 - \gamma}} \max(0,\lambda)^{\frac{1}{1 - \gamma}} + \lambda \sum_{\ell'=1}^L A_{\ell'} p(\ell,\ell') \int_{\Omega} z' m_{\ell'}(z')~dz' + \frac{z^{\alpha}}{B^{\alpha - 1}}\, ,
    \end{aligned}
\end{equation}
where $z \in \Omega$ is productivity, $m = \left(m_{\ell}\right)_{\ell = 1}^L$ is a distribution of firms in each sector and $\lambda$ is an adjoint variable. The optimal control is given by $h_{\ell}^* = \left( \frac{\gamma}{w} \max(0,\lambda) \right)^{\frac{1}{1 - \gamma}}$, for $\ell = 1, \ldots, L$. Then we define the running profit $V_{\ell}(z,t)$, for $\ell = 1, \ldots, L$, by
\begin{equation} \label{eq:mfg-value-func}
        V_{\ell}(z,t) = \sup_{h_{\ell}} \mathbb{E} \left[ \left. \int_t^\infty \left( \frac{\left(Z_{\ell}(s)\right)^{\alpha}}{B^{\alpha - 1}} - w h_{\ell}(z) \right) e^{- \rho (s - t)}~ds \right|  Z_{\ell}(t) = z \right] \, ,
\end{equation}
where $Z_{\ell}(s)$ follows~\eqref{eq:mean-field-dynamics}. If we let the equilibrium distribution be given by $m_{\ell}$ (for $\ell = 1, \ldots, L)$, then the MFG PDE system is stationary and given by
\begin{subequations} \label{eq:mfg-model}
    \begin{align}
        & V_{\ell} \in H^1(\Omega) \\
        & m_{\ell} \in H^1(\Omega) \\
        & -\frac{\sigma^2}{2} V_{\ell}'' + \rho V_{\ell} - H_{\ell} \left( z,m,V_{\ell}' \right) = 0 \label{eq:mfg-hjb} \\
        & - \frac{\sigma^2}{2} m_{\ell}'' + \left( \partial_{\lambda} H_{\ell} \left( z,m,V_{\ell}' \right) m_{\ell} \right)' = 0 \label{eq:mfg-fp} \\
        & \left. V_{\ell}' \right|_{z = 0,\bar{z}} = 0 \label{eq:mfg-hjb-bc} \\
        & \left. - \frac{\sigma^2}{2} m_{\ell}' + \partial_{\lambda} H_{\ell} \left( z,m,0 \right) m_{\ell} \right|_{z = 0,\bar{z}} = 0 \label{eq:mfg-fp-bc} \\
        & \int_{\Omega} m_{\ell}(z)~dz = 1 \, . \label{eq:mfg-fp-ic}
    \end{align}
\end{subequations}
It can be shown, using either the dynamic programming principle (c.f~\cite{Tembine2017}) or the stochastic maximum principle (c.f.~\cite{Delarue2018}), that $V_{\ell}(z)$, as defined by~\eqref{eq:mfg-value-func}, satisfies the HJB equation~\eqref{eq:mfg-hjb},~\eqref{eq:mfg-hjb-bc}. The Fokker--Planck system~\eqref{eq:mfg-fp},~\eqref{eq:mfg-fp-bc},~\eqref{eq:mfg-fp-ic} comes from the distribution in the previous section~\eqref{eq:mfg-FP-gen} and the consistency condition~\eqref{eq:mfg-consist-prob}.

\subsection{Existence and uniqueness of solutions to the MFG}
\begin{definition}
    A solution to the innovation MFG model~\eqref{eq:mfg-model} is defined to be a tuple $(m,V) = (m_1,\ldots, m_L, V_1,\ldots, V_L)$ such that $m_{\ell}:\Omega \to (0,\infty)$, $V_{\ell}: \Omega \to \R$ satisfy~\eqref{eq:mfg-model} in the weak sense for each $\ell = 1, \ldots, L$.
\end{definition}

\begin{theorem}
    There exists a solution $(m,V) \in \left[C^2(\Omega)\cap C^1\left(\bar{\Omega}\right)\right]^{2L}$ to~\eqref{eq:mfg-model}. Furthermore, if $\sum_{\ell' = 1}^L A_{\ell'} p(\ell,\ell'    )$ is small enough for every $\ell = 1, \ldots L$, then the solution is unique.
\end{theorem}
\begin{proof}[Proof outline]
   As noted in the introduction, this proof is based on the proof of existence and uniqueness in~\cite{Barker2019}. The proof presented here has some technical differences compared with the one in~\cite{Barker2019}, hence it is reproduced in full. However, it follows a similar framework and so we do not claim the proof to be new. First, for $k \in [0,\infty)$ we introduce an auxiliary system of PDEs defined by
    \begin{subequations} \label{eq:mfg-alt-model-HJB}
        \begin{align}
            & V^k \in H^1(\Omega) \label{eq:mfg-alt-vreg} \\
            & - \frac{\sigma^2}{2} \left(V^k\right)'' + \rho V^k - H^k\left(z,\left(V^k\right)'\right) = 0 \label{eq:mfg-alt-HJB} \\
            & \left. \left(V^k\right)' \right|_{z = 0,\bar{z}} = 0 \label{eq:mfg-alt-HJB-bc} \, ,
        \end{align}
    \end{subequations}
    \begin{subequations} \label{eq:mfg-alt-model-FP}
        \begin{align}
            & m^k \in H^1(\Omega) \label{eq:mfg-alt-mreg} \\
            & - \frac{\sigma^2}{2} \left(m^k\right)'' + \left( \left[ \left( \frac{\gamma}{w} \max(0,\left(V^k\right)') \right)^{\frac{\gamma}{1 - \gamma}} + k \right] m^k \right)' = 0  \label{eq:mfg-alt-FP} \\
            & \left. - \frac{\sigma^2}{2} \left(m^k\right)' + k m^k \right|_{z = 0,\bar{z}} = 0  \label{eq:mfg-alt-FP-bc} \\
            & \int_{\Omega} m^k(z)~dz = 1 \, ,  \label{eq:mfg-alt-FP-ic}
        \end{align}
    \end{subequations}
    where $H^k(z,\lambda) = (1 - \gamma)\left( \frac{\gamma}{w} \right)^{\frac{\gamma}{1 - \gamma}} \left( \max \left(0, \lambda\right) \right)^{\frac{1}{1 - \gamma}} + k \lambda + \frac{z^{\alpha}}{B^{\alpha - 1}}$. We use a modified version of upper and lower solutions (c.f.~\cite{Schmitt1978}) to prove existence and uniqueness of a weak solution $V^k$ to~\eqref{eq:mfg-alt-model-HJB} for any $k \in [0,\infty)$, and use elliptic regularity theory to show $V^k \in C^2(\Omega) \cap C^1\left(\bar{\Omega}\right)$. Next we define $m^k = \frac{1}{\left\| \bar{m}^k \right\|_1} \bar{m}^k$, where $\bar{m}^k = e^{\frac{2}{\sigma^2} \left( k z + \int_0^z \left( \frac{\gamma}{w} \max\left(0,\left(V^k\right)'\right) \right)^{\frac{\gamma}{1 - \gamma}} dy \right)}$ and $ \left\| \bar{m}^k \right\|_1 = \int_{\Omega} \bar{m}^k dz$, for $k \in [0,\infty)$. We prove that $m^k \in C^2(\Omega) \cap C^1\left(\bar{\Omega}\right)$ and that $m^k$ is the unique solution of~\eqref{eq:mfg-alt-model-FP}. Finally, we define a map $\Phi:[0,\infty)^L \to [0,\infty)^L$ by $\Phi_{\ell}(k) = \sum_{\ell' = 1}^L A_{\ell'} p(\ell,\ell') \int_{\Omega} z m^{k_{\ell'}}(z)~dz$ for$\ell = 1, \ldots, L$, and using the Brouwer fixed point theorem we prove there exists $\bar{k} \in [0,\infty)^L$ such that $\Phi\left(\bar{k}\right) = \bar{k}$. We use the contraction mapping theorem to prove uniqueness under certain smallness assumptions for the data. Then it follows, by replacing $\bar{k}_{\ell}$ with $\Phi\left(\bar{k}_{\ell}\right)$ in~\eqref{eq:mfg-alt-model-HJB} and~\eqref{eq:mfg-alt-model-FP}, that $\left(m^{\bar{k}},V^{\bar{k}}\right) = \left(m^{\bar{k}_1}, \ldots, m^{\bar{k}_L}, V^{\bar{k}_1}, \ldots, V^{\bar{k}_L}\right)$ is a (unique) solution to~\eqref{eq:mfg-model} with the required regularity.
\end{proof}

\subsubsection*{Solutions to the auxiliary HJB PDE}
\begin{theorem} \label{thm:mfg-hjb-xu}
    There exists a unique solution $V^k \in C^{2,\tau}\left(\bar{\Omega}\right)$ to the auxiliary HJB PDE~\eqref{eq:mfg-alt-model-HJB} for any $k \in [0,\infty)$ and some $\tau \in (0,1)$, where $C^{2,\tau}\left(\bar{\Omega}\right)$ is the set of $C^2$ functions on $\bar{\Omega}$ whose second derivative is H\"older continuous with exponent $\tau$. Furthermore, $0 \leq V^k \leq \frac{\bar{z}^{\alpha}}{\rho B^{\alpha - 1}}$
\end{theorem}
\begin{proof}
    The existence part of the proof uses the theory of upper and lower solutions, specifically Theorem 4.3. in~\cite{Loc2012}, and follows along similar lines to the proof of Proposition 3.12 in~\cite{Barker2019}. This shows that a solution $V^k \in W^{1,p}(\Omega)$ to the auxiliary HJB PDE exists, for some $p \geq 1$, provided the following hold true:
    
    \begin{enumerate}
        \item There exist constants $\underaccent{\bar}{V} \leq \bar{V}$ such that $\rho \underaccent{\bar}{V} - \frac{z^{\alpha}}{B^{\alpha - 1}} \leq 0 \leq \rho \bar{V} - \frac{z^{\alpha}}{B^{\alpha - 1}}$, for every $z \in \bar{\Omega}$.
        \item There exist constants $a_k \in \R$ and $b_k > 0$ such that
        \[ \left| \rho u - (1 - \gamma) \left( \frac{\gamma}{w} \right)^{\frac{\gamma}{1 - \gamma}} \left( \max(0,\lambda) \right)^{\frac{1}{1 - \gamma}} - k \lambda - \frac{z^{\alpha}}{B^{\alpha - 1}} \right| \leq a_k + b_k |\lambda|^p \, , \]
        for every $z \in \Omega$, $u \in \left[\underaccent{\bar}{V},\bar{V}\right]$ and every $\lambda \in \R$.
    \end{enumerate}
    If these two properties hold, then $\underaccent{\bar}{V} \leq V^k \leq \bar{V}$. The first assertion is true by taking $\underaccent{\bar}{V} = 0$ and $\bar{V} = \frac{\bar{z}^{\alpha}}{\rho B^{\alpha - 1}}$, which also gives the required bounds for $V^k$. The second assertion is true with $b_k = k + (1 - \gamma) \left( \frac{\gamma}{w} \right)^{\frac{\gamma}{1 - \gamma}}$, $a_k = \frac{2 \bar{z}^{\alpha}}{B^{\alpha - 1}} + b_k$, and $p = \frac{2}{1 - \gamma}$, as then
    \[ \begin{aligned}
        & \left| \rho u - (1 - \gamma) \left( \frac{\gamma}{w} \right)^{\frac{\gamma}{1 - \gamma}} \left( \max(0,\lambda) \right)^{\frac{1}{1 - \gamma}} - k \lambda - \frac{z^{\alpha}}{B^{\alpha - 1}} \right| \\
        &\leq \rho |u| + \left( k + (1 - \gamma) \left( \frac{\gamma}{w} \right)^{\frac{\gamma}{1 - \gamma}}\right) \max\left(1,|\lambda|^{\frac{1}{1 - \gamma}}\right) + \frac{\bar{z}^{\alpha}}{B^{\alpha - 1}} \leq a_k + b_k |\lambda|^p \, .
    \end{aligned} \]
    Now, since $\Omega$ is bounded and $p > 2$, $V^k \in H^1(\Omega)$. To show $V^k \in C^{2,\tau}\left(\bar{\Omega}\right)$, take any solution $V^k$ to~\eqref{eq:mfg-alt-model-HJB} and define 
    \[ f = \frac{2}{\sigma^2} \left( \left(\frac{\sigma^2}{2} - \rho \right) V^k + k \left(V^k\right)' + (1 - \gamma) \left(\frac{\gamma}{w}\right)^{\frac{\gamma}{1 - \gamma}} \left(\max(0,\left(V^k\right)'\right)^{\frac{1}{1 - \gamma}} +\frac{z^{\alpha}}{B^{\alpha - 1}}\right) \, . \]
    Then $V^k$ is a solution of $- u'' + u = f$, where $f \in L^2(\Omega)$. So, from the elliptic regularity result of Proposition 7.2. p.404 in~\cite{Taylor2011}, $V^k \in H^2(\Omega)$. Therefore $\left(V^k\right)' \in H^1(\Omega)$, and so $f \in H^1(\Omega)$ because $\alpha \in (0,1)$. So, from the elliptic regularity result of Proposition 7.4. p.407 in~\cite{Taylor2011}, $V^k \in H^3(\Omega)$. Then, by the Sobolev inequality (c.f. Theorem 6 p.270 in~\cite{Evans1998}) $V^k \in C^{2,\tau}\left(\bar{\Omega}\right)$.   
    
    To prove uniqueness we use the strong maximum principle and Hopf's lemma, as stated in~\cite{Evans1998} Section 6.4.2. pp. 330--333. Suppose, for some $k \in [0,\infty)$, there are two solutions $V_1,V_2 \in C^2(\Omega) \cap C^1(\bar{\Omega})$ to~\eqref{eq:mfg-alt-model-HJB} and $V_1 \neq V_2$. If we define $u = V_1 - V_2$, then $u$ must attain its maximum at some point $z^* \in \bar{\Omega}$. Suppose at this point $u > 0$. Note that if this were not the case, we could consider the minimum, as either its maximum or its minimum must be non--zero. The argument for the minimum is the same as the one for the maximum, so it is omitted. First suppose $z^* \in \Omega$. Since this is the maximal point, $u'(z^*) = 0$, so $V_1'(z^*) = V_2'(z^*)$. Hence, there exists an open, connected and bounded region $U$ such that $U \subset \Omega$, $z^* \in U$ and 
    \[ - \frac{\sigma^2}{2} u'' = -\rho u + ku' + (1 - \gamma) \left( \frac{\gamma}{w} \right)^{\frac{\gamma}{1 - \gamma}} \left[ \max \left( 0, V_1' \right)^{\frac{1}{1 - \gamma}} - \max \left( 0, V_2' \right)^{\frac{1}{1 - \gamma}} \right] \leq 0 \, , \]
    for every $z \in U$. So, by the strong maximum principle, $u$ is constant in $U$. In particular, using~\eqref{eq:mfg-alt-HJB}, $u(z^*) = 0$. But this is a contradiction. The only other case is $z^* \in \partial \Omega$ and $u(z) < u(z^*)$ for every $z \in \Omega$. Then, $\left. \frac{\partial u}{\partial \nu} \right|_{z^*} > 0$ by Hopf's Lemma, but by~\eqref{eq:mfg-alt-HJB-bc}, $\frac{\partial u}{\partial \nu} = \frac{\partial V_1}{\partial \nu} - \frac{\partial V_2}{\partial \nu} = 0$. This again leads to a contradiction. Therefore $V_1 = V_2$ and solutions to~\eqref{eq:mfg-alt-model-HJB} are unique for every $k \in [0,\infty)$. 
\end{proof}

\begin{proposition} \label{prop:mfg-alt-hjb-props}
   Fix $k,k_1,k_2 \in [0,\infty)$. Then, the unique classical solution to the auxiliary HJB PDE~\eqref{eq:mfg-alt-model-HJB}, as found in Theorem~\ref{thm:mfg-hjb-xu}, satisfies the following properties:
   
   \begin{enumerate}
        \item $V^k$ is an increasing function on $\bar{\Omega}$ i.e. $\left(V^k\right)' \geq 0$
        \item $\left(V^k\right)' > 0$ for all $z \in \Omega$
        \item $\|\left(V^k\right)'\|_{\infty} = \sup_{z \in \Omega} \left(V^k\right)'(z) \leq \left[ \frac{\bar{z}^{\alpha}}{(1 - \gamma)B^{\alpha - 1}} \right]^{1 - \gamma} \left( \frac{w}{\gamma} \right)^{\gamma}$
        \item $\|V^{k_1} - V^{k_2}\|_{\infty} \leq \frac{1}{\rho} \left[ \frac{\bar{z}^{\alpha}}{(1 - \gamma)B^{\alpha - 1}} \right]^{1 - \gamma} \left( \frac{w}{\gamma} \right)^{\gamma} |k_1 - k_2|$
        \item $V^k$ is strictly increasing with respect to $k$
        \item $\|\left(V^{k_1}\right)' - \left(V^{k_2}\right)'\|_{\infty} \leq \frac{4 \bar{z}}{\sigma^2} \left[ \frac{\bar{z}^{\alpha}}{(1 - \gamma)B^{\alpha - 1}} \right]^{1 - \gamma} \left( \frac{w}{\gamma} \right)^{\gamma} |k_1 - k_2|$
        \item $\left(V^k\right)''(0) > 0 > \left(V^k\right)''(\bar{z})$.
    \end{enumerate}
\end{proposition}

\begin{proof}
    Property (1): Suppose, for a contradiction, there exists $z \in \bar{\Omega}$ such that $\left(V^k\right)'(z) < 0$. First, by the boundary condition~\eqref{eq:mfg-alt-HJB-bc}, $z \in \Omega$. So, by the boundary conditions and continuity of $\left(V^k\right)'$, there exists $z_0,z_1 \in \bar{\Omega}$ with $z_0 < z_1$, $\left(V^k\right)'(z_0) = \left(V^k\right)'(z_1) = 0$ and $\left(V^k\right)'(z) \leq 0$ for all $z \in (z_0,z_1)$. Suppose that $z_0,z_1 \in \Omega$. Then $\left(V^k\right)''(z_0) \leq 0 \leq \left(V^k\right)''(z_1)$ by construction of $z_0,z_1$ and differentiability of $\left(V^k\right)'$. Furthermore,  $V^k(z_0) > V^k(z_1)$ because $\left(V^k\right)' < 0$ in $(z_0,z_1)$. So, using~\eqref{eq:mfg-alt-HJB}
    \[ 0 = - \frac{\sigma^2}{2} \left(\left(V^k\right)''(z_1) - \left(V^k\right)''(z_0) \right) + \rho \left( V^k(z_1) - V^k(z_0) \right) - \frac{1}{B^{\alpha - 1}} \left( z_1^{\alpha} - z_0^{\alpha} \right) < 0 \, . \]
    This is a contradiction, so $z_0 = 0$ or $z_1 = \bar{z}$. Assume $z_0 = 0$, we will again prove a contradiction (the other two cases of $z_1 = \bar{z}$ and both $z_0 = 0, z_1 = \bar{z}$ follow along similar arguments so their proofs are omitted). Since $\left(V^k\right)'(0) = \left(V^k\right)'(z_1) = 0$ and $\left(V^k\right)'(z) < 0$ for all $z \in (0,z_1)$ then, by continuity of $\left(V^k\right)''$, we can find $\epsilon_1,\delta_1 \in (0,\frac{z_1}{2})$ such that $\left(V^k\right)''(z) \leq 0$ for all $z \in (0,\epsilon_1]$ and $\left(V^k\right)''(z) \geq 0$ for all $z \in [z_1 - \delta_1,z_1)$. Furthermore, $V^k$ is strictly decreasing on $(z_0,z_1)$. So, using these two facts and continuity of $\left(V^k\right)'$ there exists $\delta \in (0,\delta_1]$ and $\epsilon \in (0,\epsilon_1]$ such that 
    \begin{enumerate}
        \item $\left(V^k\right)'(\epsilon) = \left(V^k\right)'(z_1 - \delta) = \min \left( \left(V^k\right)'(\epsilon_1), \left(V^k\right)'(z_1 - \delta_1) \right)$
        \item $V^k(\epsilon) > V^k(z_1 - \delta)$
        \item $\left(V^k\right)''(\epsilon) \leq 0 \leq \left(V^k\right)''(z_1 - \delta)$.
    \end{enumerate}
    Then $- \frac{\sigma^2}{2} \left(\left(V^k\right)''(z_1 - \delta) - \left(V^k\right)''(\epsilon) \right) + \rho \left( V^k(z_1 - \delta) - V^k(\epsilon) \right) - \frac{(z_1 - \delta)^{\alpha} - \epsilon^{\alpha}}{B^{\alpha - 1}} < 0$, which contradicts the fact that $V^k$ is a solution to~\eqref{eq:mfg-alt-model-HJB}. Therefore, $\left(V^k\right)' \geq 0$ in $\bar{\Omega}$.

    Property (2): From Property~(1), we know $\left(V^k\right)' \geq 0$. Now suppose, for a contradiction, there exists $z^* \in \Omega$ such that $\left(V^k\right)'(z^*) = 0$. Then $\left(V^k\right)''(z^*) = 0$, since it is a minimum of $\left(V^k\right)'$. So, by~\eqref{eq:mfg-alt-HJB}, $V^k(z^*) = \frac{z^{\alpha}}{\rho B^{\alpha - 1}}$ and, since $\left(V^k\right)'(z^*) < \frac{d}{dz} \left( \frac{z^{\alpha}}{\rho B^{\alpha - 1}} \right)$, there exists $z_0,z_1 \in \Omega$ with $z_0 < z^* < z_1$ such that
    \begin{enumerate}
        \item $\left(V^k\right)'(z_0) = \left(V^k\right)'(z_1)$
        \item $V^k (z_0) > \frac{z_0^{\alpha}}{\rho B^{\alpha - 1}}$ and $V^k (z_1) < \frac{z_1^{\alpha}}{\rho B^{\alpha - 1}}$ 
        \item $\left(V^k\right)''(z_0) \leq 0 \leq \left(V^k\right)''(z_1)$.
    \end{enumerate}
    Then $- \frac{\sigma^2}{2} \left(\left(V^k\right)''(z_1) - \left(V^k\right)''(z_0) \right) + \rho \left( V^k(z_1) - V^k(z_0) \right) - \frac{1}{B^{\alpha - 1}} \left( z_1^{\alpha} - z_0^{\alpha} \right) < 0 $, which is a contradiction of~\eqref{eq:mfg-alt-HJB}. Therefore, $\left(V^k\right)'(z) > 0$ for all $z \in \Omega$.
    
    Property (3): Since $\left(V^k\right)'$ is continuous on $\bar{\Omega}$, $\left(V^k\right)' \geq 0$ and $\left(V^k\right)'(0) = \left(V^k\right)'(\bar{z}) = 0$, then $\left(V^k\right)'$ must have a maximum that it attains at some point $z^* \in \Omega$. Furthermore, since $\left(V^k\right)'$ is continuously differentiable in $\Omega$, then $\left(V^k\right)''(z^*) = 0$. So, using the bound on $V^k$ found in Theorem~\ref{thm:mfg-hjb-xu}
    \[ \begin{aligned}
        0 \leq \left(V^k\right)'(z) & \leq \left(V^k\right)'(z^*) = \left[ \frac{w^{\frac{\gamma}{1 - \gamma}}}{\left(1 - \gamma\right) \gamma^{\frac{\gamma}{1 - \gamma}}} \left( \rho V^k(z^*) - k \left(V^k\right)'(z^*) - \frac{\left(z^*\right)^{\alpha}}{B^{\alpha - 1}} \right) \right]^{1 - \gamma} \\
        & \leq  \left[ \frac{\bar{z}^{\alpha}}{(1 - \gamma)B^{\alpha - 1}} \right]^{1 - \gamma} \left( \frac{w}{\gamma} \right)^{\gamma} \, .
    \end{aligned} \]
    
    Property (4): Take $k_1,k_2 \in [0,\infty)$ such that $k_1 < k_2$. First we show $V^{k_2} - V^{k_1} \geq 0$, then we show $V^{k_2} - V^{k_1} \leq \frac{\|\left(V^{k_1}\right)'\|_{\infty}}{\rho} (k_2 - k_1)$, and we can conclude using Property~(3). Let $u_1 = V^{k_2} - V^{k_1}$ and assume, for a contradiction, there exists $z \in \bar{\Omega}$ such that $u_1(z) < 0$. Then, $u_1$ attains a minimum at $z^* \in \bar{\Omega}$ and $u_1(z^*) < 0$. First suppose $z^* \in \Omega$, then $u_1'(z^*) = 0$ and from~\eqref{eq:mfg-alt-HJB}
    \[ - \frac{\sigma^2}{2} u_1''(z^*) = - \rho u_1(z^*) + (k_2 - k_1) \left(V^{k_1}\right)'(z^*) > 0 \, , \]
    since $\left(V^{k_1}\right)' \geq 0$ and $u_1 < 0$. Then, by continuity of $u_1''$, there exists an open bounded, connected $\Omega' \subset \Omega$ such that $z^* \in \Omega'$ and $u_1'' < 0$ for all $z \in \Omega'$. So, by the strong maximum principle, $u_1$ is constant in $\Omega'$. In particular, $u_1'' = 0$, which contradicts $u_1'' < 0$ for all $z \in \Omega'$. So, $z^* \in \partial \Omega$ and $u_1(z) < u_1(z^*)$ for all $z \in \Omega$. However, from Hopf's lemma $u_1'(z^*) \neq 0$, which contradicts~\eqref{eq:mfg-alt-HJB-bc}. So, we conclude that $u_1 \geq 0$. Now let $u_2 = V^{k_2} - V^{k_1} - \epsilon$, with $\epsilon = \frac{\|\left(V^{k_1}\right)'\|_{\infty}}{\rho} (k_2 - k_1) < \infty$. We assume, for a contradiction, there exists $z \in \bar{\Omega}$ such that $u_2(z) > 0$. Then $u_2$ attains a maximum at $z^* \in \bar{\Omega}$ and $u_2(z^*) > 0$. First suppose $z^* \in \Omega$, then $u_2'(z^*) = 0$ and from~\eqref{eq:mfg-alt-HJB}
    \[ - \frac{\sigma^2}{2} u_2''(z^*) = - \rho u_2(z^*) + (k_2 - k_1) \left(V^{k_1}\right)'(z^*) - \rho \epsilon < 0 \, , \]
    since $u_2 > 0$ and $\rho \epsilon \geq (k_1 - k_2) \left(V^{k_1}\right)'(z^*)$. Then, by continuity of $u_2''$, there exists an open bounded, connected $\Omega' \subset \Omega$ such that $z^* \in \Omega'$ and $u_2'' > 0$ for all $z \in \Omega'$. So, by the strong maximum principle, $u_2$ is constant in $\Omega'$. In particular, $u_2'' = 0$, which contradicts $u_2'' > 0$ for all $z \in \Omega'$. So, $z^* \in \partial \Omega$ and $u_2(z) > u_2(z^*)$ for all $z \in \Omega$. However, from Hopf's lemma $u_2'(z^*) \neq 0$, which contradicts~\eqref{eq:mfg-alt-HJB-bc}. So, we can conclude that $u_2 \leq 0$.
    
    Property (5): The proof of Property~(4) shows $V^k$ is increasing with respect to $k$. Now suppose, for a contradiction, there exists $z^* \in \bar{\Omega}$ such that $k_1 < k_2$ but $V^{k_1}(z^*) = V^{k_2}(z^*)$. First, assume $z^* \in \Omega$ and define $u = V^{k_2} - V^{k_1}$. Then, $u(z^*) = 0$, $u'(z^*) = 0$ and $u''(z^*) \geq 0$, since $z^*$ is a minimum of $u$. Furthermore, from Property~(2), $\left(V^{k_1}\right)'(z^*) > 0$. Therefore, using~\eqref{eq:mfg-alt-HJB}, we get the contradiction
    \[ 0 = - \frac{\sigma^2}{2} u'' + (k_1 - k_2) \left(V^{k_1}\right)' < 0 \, . \]
    Hence, $z^* \in \partial \Omega$, so $u'(z^*) = 0$ ,using~\eqref{eq:mfg-alt-HJB-bc}. But, $u'(z^*) \neq 0$ by Hopf's lemma, which is a contradiction. So, $V^k$ is strictly increasing with respect to $k$.
    
    Property (6):
    Fix $k_1, k_2 \in [0,\infty)$. Let $u = V^{k_1} - V^{k_2}$. Then, $u$ satisfies
    \[ \frac{\sigma^2}{2} u'' = \rho u - k_1 u' + (k_2 - k_1) \left(V^{k_2}\right)' - (1 - \gamma) \left( \frac{\gamma}{w} \right)^{\frac{\gamma}{1 - \gamma}} \left( \left( \left(V^{k_1}\right)' \right)^{\frac{1}{1 - \gamma}} - \left( \left(V^{k_2}\right)' \right)^{\frac{1}{1 - \gamma}} \right) \, . \]
    Suppose for $z \in \Omega$, $u'(z) \geq 0$. Then, since $u'(0) = 0$, there exists $z_0 \in [0,z]$ such that $u'(y) \geq 0$ for all $y \in [z_0,z]$ and $u'(z_0) = 0$. Therefore
    \[ \begin{aligned}
        0 \leq u'(z) = \int_{z_0}^{z} u''(y)~dy & \leq \frac{2 \bar{z}}{\sigma^2} \left( \rho ||u||_{\infty} + |k_2 - k_1| \left| \left|  \left(V^{k_2}\right)' \right| \right|_{\infty} \right) \\ 
        & \leq \frac{4 \bar{z}}{\sigma^2} \left[ \frac{\bar{z}^{\alpha}}{(1 - \gamma)B^{\alpha - 1}} \right]^{1 - \gamma} \left( \frac{w}{\gamma} \right)^{\gamma} |k_2 - k_1| \, .
    \end{aligned} \]
    We can similarly show that $u'(z) \geq - \frac{4 \bar{z}}{\sigma^2} \left[ \frac{\bar{z}^{\alpha}}{(1 - \gamma)B^{\alpha - 1}} \right]^{1 - \gamma} \left( \frac{w}{\gamma} \right)^{\gamma} |k_2 - k_1|$ if $u'(z) \leq 0$. Hence, $\left(V^k\right)'$ is Lipschitz continuous with respect to $k$ with the required constant.
    
    Property (7): First, we will show $\left(V^k\right)''(0) > 0$ and then we will show $\left(V^k\right)''(\bar{z}) < 0$. Both steps use a similar method. Note that $\left(V^k\right)''(0) \geq 0$ and $\left(V^k\right)''(\bar{z}) \leq 0$, because ${\left(V^k\right)'(0) = \left(V^k\right)'(\bar{z}) = 0}$ and $\left(V^k\right)(z) > 0$ for all $z \in (0,\bar{z})$. So suppose, for a contradiction, that $\left(V^k\right)''(0) = 0$. Then, since $V^k \in C^{2,\tau}\left(\bar{\Omega}\right)$, we can use continuity of $V^k, \left(V^k\right)', \left(V^k\right)''$ and~\eqref{eq:mfg-alt-HJB},~\eqref{eq:mfg-alt-HJB-bc} to show $V^k(0) = 0$. We can also use continuity of $\left(V^k\right)''$ to show that for every $C > 0$ there exists $\epsilon_1 > 0$ such that $z \in (0,\epsilon_1) \implies \left(V^k\right)''(z) < C$. Therefore, for any $z \in (0,\epsilon_1)$
    \begin{equation} \label{eq:mfg-contra-Vbound}
        V^k(z) = \int_0^z \int_0^y \left(V^k\right)''(y')~dy'~dy - z \left(V^k\right)'(0) - V^k(0) < \frac{C}{2} z^2 \, .
        \end{equation}
    Since $\left(V^k\right)' > 0$ in $\Omega$, there exists $\epsilon_2 > 0$ such that $\left(V^k\right)'$ increases on $(0,\epsilon_2)$. Therefore, $\left(V^k\right)'' \geq 0$ on $(0,\epsilon_2)$. Take $C = \frac{2}{\rho B^{\alpha - 1}}$ and $\epsilon = \min(\epsilon_1,\epsilon_2,1)$. Then, from~\eqref{eq:mfg-contra-Vbound}
    \begin{equation} \label{eq:mfg-contra-Vbound2}
        V^k(z) < \frac{z^2}{\rho B^{\alpha - 1}} \leq \frac{z^{\alpha}}{\rho B^{\alpha - 1}} \, ,
    \end{equation}
    for all $z \in (0,\epsilon)$. But, by rearranging~\eqref{eq:mfg-alt-HJB} and using $\left(V^k\right)',\left(V^k\right)'' \geq 0$, we get $V^k(z) \geq \frac{z^{\alpha}}{\rho B^{\alpha - 1}}$, which contradicts~\eqref{eq:mfg-contra-Vbound2}. Hence, $\left(V^k\right)''(0) > 0$. Now suppose, for a contradiction, that $\left(V^k\right)''(\bar{z}) = 0$. Then, since $V^k \in C^{2,\tau}\left(\bar{\Omega}\right)$, we can use continuity of $V^k, \left(V^k\right)', \left(V^k\right)''$ and~\eqref{eq:mfg-alt-HJB},~\eqref{eq:mfg-alt-HJB-bc} to show $V^k(\bar{z}) = \frac{\bar{z}^{\alpha}}{\rho B^{\alpha - 1}}$. We can also use continuity of $\left(V^k\right)''$ to show that for every $C > 0$ there exists $\epsilon_1 > 0$ such that $z \in (\bar{z} - \epsilon_1, \bar{z}) \implies \left(V^k\right)''(z) > -C$. Therefore, for any $z \in (\bar{z} - \epsilon_1, \bar{z})$
    \begin{equation} \label{eq:mfg-contra-Vbound3} 
        V^k(z) = \int_z^{\bar{z}} \int_y^{\bar{z}} \left(V^k\right)''(y)~dy + V^k(\bar{z}) -\left(\bar{z} - z\right) \left(V^k\right)'(\bar{z})  > \frac{\bar{z}^{\alpha}}{\rho B^{\alpha - 1}} - \frac{C}{2} (\bar{z} - z)^2 \, .
    \end{equation}
    Since $\left(V^k\right)' > 0$ in $\Omega$, there exists $\epsilon_2 > 0$ such that $\left(V^k\right)'$ decreases on $(\bar{z} - \epsilon_2, \bar{z})$. Therefore, $\left(V^k\right)'' \leq 0$ on $(\bar{z} - \epsilon_2, \bar{z})$. Take $C > 0$ such that
    \[ \frac{k C}{\rho} + \frac{C}{2} + \frac{C^{\frac{1}{1 - \gamma}}}{\rho} (1 - \gamma) \left( \frac{\gamma}{w} \right)^{\frac{\gamma}{1 - \gamma}} \leq \alpha \frac{\bar{z}^{\alpha - 1}}{{\rho B^{\alpha - 1}}} \, , \]
    and $\epsilon = \min(\epsilon_1,\epsilon_2,1)$. Then, from~\eqref{eq:mfg-contra-Vbound3} 
    \begin{equation} \label{eq:mfg-contra-Vbound4}
        V^k(z) > \frac{\bar{z}^{\alpha}}{\rho B^{\alpha - 1}} - \frac{C}{2} (\bar{z} - z)^2 \geq \frac{\bar{z}^{\alpha}}{\rho B^{\alpha - 1}} - \frac{C}{2} (\bar{z} - z) \, ,
    \end{equation}
    for all $z \in (\bar{z} - \epsilon, \bar{z})$. But, from~\eqref{eq:mfg-alt-HJB} and using $\left(V^k\right)' \leq C (\bar{z} - z)$, $\left(V^k\right)'' \leq 0$, we get
    \[ \begin{aligned}
            V^k(z) & = \frac{1}{\rho} \left( \frac{\sigma^2}{2} \left(V^k\right)''(z) + k \left(V^k\right)'(z) + (1 - \gamma) \left( \frac{\gamma}{w} \right)^{\frac{\gamma}{1 - \gamma}} \left( \left(V^k\right)' \right)^{\frac{1}{1 - \gamma}} + \frac{z^{\alpha}}{B^{\alpha - 1}} \right) \\
            & \leq \frac{k C}{\rho} (\bar{z} - z) + \frac{C^{\frac{1}{1 - \gamma}}}{\rho} (1 - \gamma) \left( \frac{\gamma}{w} \right)^{\frac{\gamma}{1 - \gamma}} (\bar{z} - z)^{\frac{1}{1 - \gamma}} + \frac{z^{\alpha}}{\rho B^{\alpha - 1}} \\
            & \leq \frac{z^{\alpha}}{\rho B^{\alpha - 1}} + \left( \frac{k C}{\rho} + \frac{C^{\frac{1}{1 - \gamma}}}{\rho} (1 - \gamma) \left( \frac{\gamma}{w} \right)^{\frac{\gamma}{1 - \gamma}} \right) (\bar{z} - z) \, .
        \end{aligned} \]
    So, using~\eqref{eq:mfg-contra-Vbound4}, we get $\alpha \frac{\bar{z}^{\alpha - 1}}{\rho B^{\alpha - 1}} (\bar{z} - z) <  \left( \frac{k C}{\rho} + \frac{C^{\frac{1}{1 - \gamma}}}{\rho} (1 - \gamma) \left( \frac{\gamma}{w} \right)^{\frac{\gamma}{1 - \gamma}} \right) (\bar{z} - z)$, which contradicts the definition of $C$. Hence, $\left(V^k\right)''(\bar{z}) < 0$.
\end{proof}

\subsubsection*{The auxiliary Fokker--Planck equation}

\begin{definition} \label{def:mfg-alt-FPsol}
    Fix $k \in [0,\infty)$ and let $V^k \in C^2(\Omega) \cap C^1\left(\bar{\Omega}\right)$ denote the unique solution to~\eqref{eq:mfg-alt-model-HJB}. Then, we define the function $m^k: \Omega \to (0,\infty)$ by
    \begin{subequations} \label{eq:mfg-alt-FPsol}
        \begin{equation}
            \bar{m}^k = e^{\frac{2}{\sigma^2} \left( k z + \int_0^z \left( \frac{\gamma}{w} \left(V^k\right)' \right)^{\frac{\gamma}{1 - \gamma}}~dy \right)}
        \end{equation}
        \begin{equation}
           \left\| \bar{m}^k \right\|_1 = \int_{\Omega} \bar{m}^k dz
        \end{equation}
        \begin{equation}
            m^k = \frac{1}{\left\| \bar{m}^k \right\|_1} \bar{m}^k \, .
        \end{equation}
    \end{subequations}
\end{definition}

\begin{proposition} \label{prop:mfg-alt-FPsol}
    For every $k \in [0,\infty)$, $m^k \in C^2(\Omega) \cap C^1\left(\bar{\Omega}\right)$ where $m^k$ is defined by~\eqref{def:mfg-alt-FPsol}. 
\end{proposition}

\begin{proof}
    First, note that $m^k$ is well defined because $\left(V^k\right)' \geq 0$ and $\left(V^k\right)'$ is uniformly bounded. Hence, there exists $C \in (1,\infty)$ such that $\bar{m}^k(z) \in [1,C]$ and $\left| \left| \bar{m}^k \right| \right|_1 \in [\bar{z},C \bar{z}]$, so $m^k(z) \in \left[\frac{1}{C \bar{z}}, \frac{C}{\bar{z}}\right]$. Furthermore, $m^k \in C\left(\bar{\Omega}\right)$ because $V^k \in C^1\left(\bar{\Omega}\right)$. Now, if $m^k \in C^2(\Omega) \cap C^1\left(\bar{\Omega}\right)$, then its derivatives would be
    \begin{subequations} \label{eq:mfg-alt-FP-derivs}
        \begin{equation}
            \left(m^k\right)' = \frac{2}{\sigma^2} \left(k + \left( \frac{\gamma}{w} \left(V^k\right)' \right)^{\frac{\gamma}{1 - \gamma}} \right) m^k
        \end{equation}
        \begin{equation}
            \begin{aligned}
                \left(m^k\right)'' =& \frac{2}{\sigma^2} \left(k + \left( \frac{\gamma}{w} \left(V^k\right)' \right)^{\frac{\gamma}{1 - \gamma}} \right) \left(m^k\right)' \\
                &+ \frac{2 \gamma^2}{\sigma^2 w (1 - \gamma)} \left( \frac{\gamma}{w} \left(V^k\right)' \right)^{\frac{2\gamma - 1}{1 - \gamma}} \left(V^k\right)'' m^k \, . 
            \end{aligned}
        \end{equation}
    \end{subequations}
    But, since $V^k \in C^1\left(\bar{\Omega}\right)$ and $m^k \in C\left(\bar{\Omega}\right)$, then $\frac{2}{\sigma^2} \left(k + \left( \frac{\gamma}{w} \left(V^k\right)' \right)^{\frac{\gamma}{1 - \gamma}} \right) m^k$ is well--defined and continuous for all $z \in \bar{\Omega}$. Hence, $m^k \in C^1\left(\bar{\Omega}\right)$. Then, $\left(m^k\right)',\left(V^k\right)$ and $\left(V^k\right)''$ are continuous in $\Omega$ and from Proposition~\ref{prop:mfg-alt-hjb-props} $\left(V^k\right)' > 0$ in $\Omega$. Hence, $\left(m^k\right)''$ is well--defined in $\Omega$, $\left(m^k\right)'' \in C(\Omega)$ and $m^k \in C^2(\Omega) \cap C^1\left(\bar{\Omega}\right)$.
\end{proof}

\begin{theorem} \label{thm:mfg-alt-FP-xu}
   There exists a unique solution $m^k \in C^2(\Omega) \cap C^1\left(\bar{\Omega}\right)$ to the auxiliary Fokker--Planck PDE~\eqref{eq:mfg-alt-model-FP} for any $k \in [0,\infty)$.
\end{theorem}

\begin{proof}
    Take $m^k$ defined in Definition~\ref{def:mfg-alt-FPsol}. Then, $m^k \in C^2(\Omega) \cap C^1\left(\bar{\Omega}\right)$ by Proposition~\ref{prop:mfg-alt-FPsol}. Furthermore, from~\eqref{eq:mfg-alt-FP-derivs}, $m^k$ satisfies~\eqref{eq:mfg-alt-FP},~\eqref{eq:mfg-alt-FP-bc}. Finally, by construction, $m^k$ satisfies~\eqref{eq:mfg-alt-FP-ic}. Therefore, a solution to the auxiliary Fokker--Planck equation~\eqref{eq:mfg-alt-model-FP} exists, it is given by $m^k$, and $m^k \in C^2(\Omega) \cap C^1\left(\bar{\Omega}\right)$. To prove uniqueness we follow the same proof in~\cite{Barker2019}. For brevity we only outline the argument here. First, with $\bar{m}^k$ defined as in~\eqref{eq:mfg-alt-FPsol}, we can use regularity of $\bar{m}^k$ from Proposition~\ref{prop:mfg-alt-FPsol} to show~\eqref{eq:mfg-alt-model-FP} is equivalent to
    \begin{subequations} \label{eq:mfg-alt-FPuniq1}
        \begin{align}
            & m^k, \frac{m^k}{\bar{m}^k} \in H^1(\Omega) \\
            & \left( \bar{m}^k\left(\frac{m^k}{\bar{m}^k}\right)'\right)' = 0 \label{eq:mfg-alt-FPuniq1c} \\
            & \left. \bar{m}^k\left(\frac{m^k}{\bar{m}^k}\right)' \right|_{z = 0,\bar{z}} = 0 \, , \quad \int_{\Omega} m^k~dz = 1 \, .
        \end{align}
    \end{subequations}
    Then, by multiplying~\eqref{eq:mfg-alt-FPuniq1c} by $\frac{m^k}{\bar{m}^k}$, integrating over $\Omega$ and using integration by parts, the system~\eqref{eq:mfg-alt-FPuniq1} is equivalent to
    \begin{equation} \label{eq:mfg-alt-FPuniq2}
        m^k \in H^1(\Omega) \, , \quad \text{there exists } Z > 0 \text{ such that } m^k = \frac{1}{Z} \bar{m}^k \, , \quad \int_{\Omega} m^k~dz = 1 \, .
    \end{equation}
    From the previous results in this section, we have shown there exists a unique solution to~\eqref{eq:mfg-alt-FPuniq2} given by $m^k$ from Definition~\ref{def:mfg-alt-FPsol}. Hence, existence and uniqueness of the auxiliary Fokker--Planck PDE follows from the equivalence between~\eqref{eq:mfg-alt-model-FP} and~\eqref{eq:mfg-alt-FPuniq2}. 
\end{proof}

\subsubsection*{The fixed point problem}

\begin{definition} \label{def:mfg-phi}
    Fix $k = \left(k_{\ell}\right)_{\ell = 1}^L \in [0,\infty)^L$. For $\ell = 1, \ldots, L$, let $V^{k_{\ell}}$ be the unique solution to the auxiliary HJB PDE~\eqref{eq:mfg-alt-model-HJB} with constant $k_{\ell}$, and let $m^{k_{\ell}}$ be the unique solution to the auxiliary Fokker--Planck PDE~\eqref{eq:mfg-alt-model-FP} with constant $k_{\ell}$. Then we define the function $\Phi:[0,\infty)^L \to [0,\infty)^L$ by
    \[ \Phi_{\ell}(k) = \sum_{\ell'=1}^L A_{\ell'} p(\ell,\ell') \int_0^{\bar{z}} z m^{k_{\ell'}}(z)~dz \, , \quad \ell = 1, \ldots, L \, . \]
\end{definition}

\begin{proposition} \label{prop:mfg-phi-bound}
    The function $\Phi$ defined in Definition~\ref{def:mfg-phi} is bounded. Furthermore, defining $P$ as the $L \times L$ matrix with entries $P_{\ell,\ell'} = p(\ell,\ell')$ and $A$ as the column vector $(A_1,\ldots, A_L)^T$, then
    \[ 0 \leq \left\|\Phi(k)\right\|_1 \leq \bar{z} \left\| PA \right\|_1 \, , \]
    where the 1--norm $\|\cdot\|_1$ is defined as $\|x\|_1 = \sum_{\ell = 1}^L |x_{\ell}|$ for any $x \in \R^L$.
\end{proposition}

\begin{remark}
    Due to this proposition, we can define $\zeta = \bar{z} \left\| PA \right\|_1$ and consider only the restriction of $\Phi$ to $[0,\zeta]^L$, which we will still denote by $\Phi$ for convenience.
\end{remark}

\begin{proof}
    Take $\ell = 1, \ldots, L$. Then $\Phi_{\ell}(k) = \sum_{\ell'=1}^L A_{\ell'} p(\ell,\ell') \int_0^{\bar{z}} z m^{k_{\ell'}}(z)~dz \geq 0$, since $p(\ell,\ell') \geq 0$ and $m^{k_{\ell'}} \geq 0$. Similarly, since $m^{k_{\ell'}}$ is a probability distribution, $\Phi_{\ell}(k) \leq \bar{z} \sum_{\ell'=1}^L A_{\ell'} p(\ell,\ell') \int_0^{\bar{z}} m^{k_{\ell'}}(z)~dz = \bar{z} \sum_{\ell'=1}^L A_{\ell'} p(\ell,\ell')$. Therefore,
    \[ 0 \leq \sum_{\ell = 1}^L \Phi_{\ell}(k) \leq \bar{z} \sum_{\ell = 1}^L \sum_{\ell' = 1}^L A_{\ell'} p(\ell,\ell') = \bar{z}\left\| PA \right\|_1 \, . \]
\end{proof}

\begin{theorem} \label{thm:mfg-phi-lip}
     The function $\Phi:[0,\zeta]^L \to [0,\zeta]^L$ defined in Definition~\ref{def:mfg-phi} is Lipschitz in the 1--norm on $\R^L$. The Lipschitz constant is given by $\underaccent{\bar}{C}~\max_{\ell = 1, \ldots, L} A_{\ell} P_{\ell}$, where $P_{\ell} = \sum_{\ell' = 1}^L p(\ell',\ell)$ and $\underaccent{\bar}{C}$ depends on $\left\|PA\right\|_1$, but not explicitly on $P$ or $A$.
\end{theorem}

\begin{proof}
    First, fix $k \in [0,\zeta]$. From Property~(5) of Proposition~\ref{prop:mfg-alt-hjb-props}, the continuity of $V^k,\left(V^k\right)',\left(V^k\right)''$ with respect to $z$ in $\bar{\Omega}$, and equations~\eqref{eq:mfg-alt-HJB},~\eqref{eq:mfg-alt-HJB-bc}, we find
    \[ \left(V^k\right)''(0) = \frac{2 \rho}{\sigma^2} V^k(0) \geq \frac{2 \rho}{\sigma^2} V^0(0) = \left(V^0\right)''(0) > 0 \, , \]
     with the middle inequality an equality if and only if $k = 0$. Similarly, $\left(V^k\right)''(\bar{z}) \leq \left(V^{\zeta}\right)'' (\bar{z}) < 0$ with the first inequality an equality if and only if $k = 0$. Moreover, $\left(V^k\right)''$ is continuous with respect to $k$ due to~\eqref{eq:mfg-alt-HJB} and continuity of $V^k,\left(V^k\right)'$ with respect to $k$, which was proven in Proposition~\ref{prop:mfg-alt-hjb-props}. Therefore, there exists $\epsilon_1,\epsilon_2 \in (0,1)$ and $C_1,C_2 > 0$, independent of $k$, such that
    \[ \begin{aligned}
        & \left(V^k\right)'(z) = \int_0^z \left(V^k\right)''(y)~dy \geq \int_0^z \left(V^0\right)''(y)~dy \geq C_1 z \, , & \quad &\text{if} \, z \in [0,\epsilon_1] \\
        & \left(V^k\right)'(z) = - \int_z^{\bar{z}} \left(V^k\right)''(y)~dy \geq - \int_z^{\bar{z}} \left(V^{\zeta}\right)''(y)~dy \geq C_2 z \, , & \quad &\text{if} \, z \in [\bar{z} - \epsilon_2,\bar{z}] \, .
    \end{aligned} \]
    Furthermore, by continuity of $\left(V^k\right)'$ with respect to $k$ and compactness of $[0,\zeta]$, there exists $C_3 > 0$ such that $\inf_{k \in [0,\zeta]} \left(V^k\right)'(z) \geq C_3$ if $z \in [\epsilon_1, \bar{z} - \epsilon_2]$. Note that $C_j$ for $j = 1,2,3$ are all independent of $k \in [0,\zeta]$. Therefore, if $\gamma \leq \frac{1}{2}$, for any $k_1,k_2 \in [0,\zeta]$:
    \begin{equation} \label{eq:mfg-int-est-1}
        \begin{aligned}
            \int_0^{\bar{z}} \bigg[ & \min \left( \left(V^{k_1}\right)'(z),\left(V^{k_2}\right)'(z) \right) \bigg]^{\frac{2 \gamma - 1}{1 - \gamma}}~dz \\
            & \leq \int_0^{\epsilon_1} \left( C_1 z \right)^{\frac{2 \gamma - 1}{1 - \gamma}}~dz + \int_{\epsilon_1}^{\bar{z} - \epsilon_2} C_3^{\frac{2 \gamma - 1}{1 - \gamma}}~dz + \int_{\bar{z} - \epsilon_2}^{\bar{z}} \left( C_2 (\bar{z} - z) \right)^{\frac{2 \gamma - 1}{1 - \gamma}}~dz \\
            & \leq \frac{1 - \gamma}{\gamma} (C_1^{\frac{2 \gamma - 1}{1 - \gamma}} + C_2^{\frac{2 \gamma - 1}{1 - \gamma}}) + C_3^{\frac{2 \gamma - 1}{1 - \gamma}} \bar{z} \, ,
        \end{aligned}
    \end{equation}
    while, using Proposition~\ref{prop:mfg-alt-hjb-props}, if $\gamma \geq \frac{1}{2}$
    \begin{equation} \label{eq:mfg-int-est-2}
        \begin{aligned}
            \int_0^{\bar{z}} \bigg[ \max \left( \left(V^{k_1}\right)'(z),\left(V^{k_2}\right)'(z) \right) & \bigg]^{\frac{2 \gamma - 1}{1 - \gamma}}~dz \\ 
            & \leq \bar{z} \left( \frac{\bar{z}^{\alpha}}{(1 - \gamma) B^{\alpha - 1}}\right)^{2 \gamma - 1} \left(\frac{w}{\gamma}\right)^{\frac{\gamma (2 \gamma - 1)}{1 - \gamma}} \, .
        \end{aligned}
    \end{equation}
    Now, with the definition of $\bar{m}^k$ in~\eqref{eq:mfg-alt-FPsol}, for any $k_1,k_2 \in [0,\zeta]$ we have
    \[ \begin{aligned}
            |\bar{m}^{k_1} - \bar{m}^{k_2}| & = \left| e^{\frac{2}{\sigma^2} \left( k_1 z + \int_0^z \left( \frac{\gamma}{w} \left(V^{k_1}\right)'(y) \right)^{\frac{\gamma}{1 - \gamma}}~dy \right)} - e^{\frac{2}{\sigma^2} \left( k_2 z + \int_0^z \left( \frac{\gamma}{w} \left(V^{k_2}\right)'(y) \right)^{\frac{\gamma}{1 - \gamma}}~dy \right)} \right| \\
            & = \left| \frac{2}{\sigma^2} \int_{k_2 z + \int_0^z \left( \frac{\gamma}{w} \left(V^{k_2}\right)'(y) \right)^{\frac{\gamma}{1 - \gamma}}~dy}^{k_1 z + \int_0^z \left( \frac{\gamma}{w} \left(V^{k_1}\right)'(y) \right)^{\frac{\gamma}{1 - \gamma}}~dy} e^{\frac{2}{\sigma^2} u}~du \right| \, .
        \end{aligned} \]
    Then, using the uniform bound on $\left(V^k\right)'(y)$ with respect to $k$ given by Proposition~\ref{prop:mfg-alt-hjb-props}, we get
    \[ \begin{aligned}
        |\bar{m}^{k_1} - \bar{m}^{k_2}| & \leq \frac{2 \bar{C}_1}{\sigma^2} \left| (k_1 - k_2) z + \int_0^z \left[ \left( \frac{\gamma}{w} \left(V^{k_1}\right)'(y) \right)^{\frac{\gamma}{1 - \gamma}} - \left( \frac{\gamma}{w} \left(V^{k_2}\right)'(y) \right)^{\frac{\gamma}{1 - \gamma}} \right]~dy \right| \\
        & \leq \frac{2 \bar{C}_1}{\sigma^2} \left( |k_1 - k_2| z + \left( \frac{\gamma}{w} \right)^{\frac{\gamma}{1 - \gamma}} \int_0^{\bar{z}} \int_{\left(V^{k_2}\right)'(y)}^{\left(V^{k_1}\right)'(y)} \frac{\gamma}{1 - \gamma} u^{\frac{2 \gamma - 1}{1 - \gamma}}~du~dy \right) \, , 
    \end{aligned} \]
    where $\bar{C}_1 = e^{\frac{2 \bar{z}}{\sigma^2} \left( \zeta +  \left[\frac{\gamma \bar{z}^{\alpha}}{(1 - \gamma) w B^{\alpha - 1}} \right]^{\gamma} \right)}$. Then, using Proposition~\ref{prop:mfg-alt-hjb-props} and either~\eqref{eq:mfg-int-est-1} or~\eqref{eq:mfg-int-est-2}, we get
    \begin{equation} \label{eq:mfg-lip-est1}
        \begin{aligned}
            |\bar{m}^{k_1} - \bar{m}^{k_2}| \leq & \frac{2 \bar{C}_1}{\sigma^2} \Bigg( |k_1 - k_2| z + \left( \frac{\gamma}{w} \right) ^{\frac{\gamma}{1 - \gamma}} \frac{\gamma}{1 - \gamma} ||\left(V^{k_1}\right)' - \left(V^{k_2}\right)'||_{\infty} \\
            & \quad \quad \quad \quad \, \, \int_0^z \max \Big[ \left( \left(V^{k_1}\right)'(y)\right)^{\frac{2 \gamma - 1}{1 - \gamma}},\left(\left(V^{k_2}\right)'(y) \right)^{\frac{2 \gamma - 1}{1 - \gamma}}\Big]~dy \Bigg) \\
            \leq & \frac{2 \bar{C}_1}{\sigma^2} \left( z + \bar{C}_2 \right) |k_1 - k_2|\, ,
        \end{aligned}
    \end{equation}
    where $\bar{C}_2 = \frac{4 \bar{z}}{\sigma^2} \left( \frac{\bar{z}}{(1 - \gamma) B^{\alpha - 1}} \right)^{1 - \gamma} \left( \frac{\gamma}{w} \right)^{\frac{\gamma^2}{1 - \gamma}} \left( C_1^{\frac{2 \gamma - 1}{1 - \gamma}} + C_2^{\frac{2 \gamma - 1}{1 - \gamma}} + \frac{\gamma}{1 - \gamma} C_3^{\frac{2 \gamma - 1}{1 - \gamma}} \bar{z} \right)$, if $\gamma < \frac{1}{2}$. While $\bar{C}_2 = \frac{\gamma}{1 - \gamma} \frac{4 \bar{z}^2}{\sigma^2} \left( \frac{w \bar{z}^{\alpha}}{\gamma(1 - \gamma)B^{\alpha - 1}}\right)^{\gamma}$, if $\gamma \geq \frac{1}{2}$. Note that for any $k \in [0,\zeta]$, $\left\| \bar{m}^k \right\|_1$ satisfies
    \begin{equation} \label{eq:mfg-lip-est2}
            \left| \left| \bar{m}^k \right| \right|_1 = \int_0^{\bar{z}} e^{\frac{2}{\sigma^2} \left[ k z + \int_0^z \left( \frac{\gamma}{w} \left(V^k\right)'(y) \right)^{\frac{\gamma}{1 - \gamma}}~dy \right]}~dz \geq 1 \, ,
    \end{equation}
    as $\left(V^k\right)' \geq 0$. So, for any $k_1, k_2 \in [0,\zeta]$, using~\eqref{eq:mfg-lip-est1} and~\eqref{eq:mfg-lip-est2}, we have
    \begin{equation} \label{eq:mfg-lip-est3}
        \begin{aligned}
            \bigg| \int_0^{\bar{z}} z (m^{k_1} - m^{k_2})~dz \bigg| \leq & \frac{1}{\left\| \bar{m}^{k_1} \right\|_1} \left| \int_0^{\bar{z}} z (\bar{m}^{k_1} - \bar{m}^{k_2})~dz \right| \\
            & + \int_0^{\bar{z}} z \frac{\bar{m}^{k_2}}{\left\| \bar{m}^{k_1} \right\|_1 \left\| \bar{m}^{k_2} \right\|_1}~dz \Big| \left\| \bar{m}^{k_1} \right\|_1 - \left| \bar{m}^{k_2} \right\|_1 \Big| \\
            \leq & 2 \bar{z} \int_0^{\bar{z}} |\bar{m}^{k_1} - \bar{m}^{k_2}|~dz \leq \frac{4 \bar{C}_1 \bar{z}}{\sigma^2} \int_0^{\bar{z}} (z + \bar{C}_2)~dz |k_1 - k_2| \\
            = & \frac{2 \bar{C}_1 \bar{z}^2 (\bar{z} + 2 \bar{C}_2)}{\sigma^2} |k_1 - k_2| := \underaccent{\bar}{C}~|k_1 - k_2| \, .
            \end{aligned}
    \end{equation}
    Now take $k^{(1)},k^{(2)} \in [0,\zeta]^L$. Define $P_{\ell} = \sum_{\ell' = 1}^L p(\ell',\ell)$, then recalling the definition of $\Phi$ given in Definition~\ref{def:mfg-phi} and using~\eqref{eq:mfg-lip-est3}
    \begin{equation} \label{eq:mfg-lip}
        \begin{aligned}
            \left\|\Phi(k^{(1)}) - \Phi(k^{(2)})\right\|_1 = \sum_{\ell = 1}^L \sum_{\ell' = 1}^L A_{\ell'} p(\ell,\ell') \left| \int_0^{\bar{z}} z (m^{k^{(1)}_{\ell'}} - m^{k^{(2)}_{\ell'}})~dz \right| \\
            \leq \underaccent{\bar}{C}~\sum_{\ell' = 1}^L A_{\ell'} P_{\ell'} \left| k^{(1)}_{\ell'} - k^{(2)}_{\ell'} \right| \leq \underaccent{\bar}{C}~\max_{\ell = 1, \ldots, L} A_{\ell} P_{\ell} \left\|k^{(1)} - k^{(2)} \right\|_1 \, ,
        \end{aligned}
    \end{equation}
    which concludes the proof.
\end{proof}

\begin{theorem} \label{thm:mfg-xu}
    For any given data, there exists a solution to the innovation MFG~\eqref{eq:mfg-model}. Furthermore, if $\|PA\|_1$ is fixed, this solution is unique provided $A_{\ell} P_{\ell} < \frac{1}{\underaccent{\bar}{C}}$ for every $\ell = 1, \ldots, L$.
\end{theorem}

\begin{proof}
	From Proposition~\ref{prop:mfg-phi-bound} and Theorem~\ref{thm:mfg-phi-lip}, the function $\Phi: [0,\zeta]^L \to [0,\zeta]^L$ is a continuous function from a convex compact subset of $\R^L$ to itself. Therefore, by Brouwer’s fixed point theorem, $\Phi$ has a fixed point. Furthermore, Theorem~\ref{thm:mfg-phi-lip} shows that $\Phi$ is a Lipschitz function in $\|\cdot\|_1$. The Lipschitz constant is given by $\underaccent{\bar}{C}~\max_{\ell = 1, \ldots, L} A_{\ell} P_{\ell}$, where $\underaccent{\bar}{C}$ depends on $\|PA\|_1$ but not directly on $P_{\ell}$ or $A_{\ell}$. Therefore, for fixed $\|PA\|_1$, $\Phi$ is a contraction map provided $A_{\ell} P_{\ell} < \frac{1}{\underaccent{\bar}{C}}$ for every $\ell = 1, \ldots, L$, and in this case the fixed point is unique.

    Theorems~\ref{thm:mfg-hjb-xu} and~\ref{thm:mfg-alt-FP-xu} proved existence and uniqueness of solutions to equations~\eqref{eq:mfg-alt-model-HJB} and~\eqref{eq:mfg-alt-model-FP} respectively for any $k \in [0,\zeta]$. Now, if $k^*$ is a fixed point of $\Phi$ then $(m^*,V^*) := \left(m^{k^*_{\ell}},V^{k^*_{\ell}}\right)_{\ell = 1}^L$ is a solution to~\eqref{eq:mfg-model}, which can be seen by replacing $k^*_{\ell}$ with $\Phi_{\ell}\left(k^*\right)$ in~\eqref{eq:mfg-alt-model-HJB},~\eqref{eq:mfg-alt-model-FP} for every $\ell = 1, \ldots, L$. Conversely, if $(m^*,V^*)$ is a solution to~\eqref{eq:mfg-alt-model-HJB},~\eqref{eq:mfg-alt-model-FP}, then clearly, by defining $k^*$ co--ordinate wise as $k^*_{\ell} = \sum_{\ell' = 1}^L A_{\ell'} p(\ell,\ell') \int_0^{\bar{z}} z m_{\ell'}(z)~dz$, $k^* \in [0,\zeta]^L$ is a fixed point of $\Phi$. Furthermore, by uniqueness of~\eqref{eq:mfg-alt-model-HJB},~\eqref{eq:mfg-alt-model-FP}, $\left(m^{k^*},V^{k^*}\right) = (m^*,V^*)$. So, existence and uniqueness of solutions to the innovation MFG~\eqref{eq:mfg-model} is equivalent to existence and uniqueness of fixed points of $\Phi$. Hence, there exists a solution to the innovation MFG. Furthermore, this solution is unique, provided $A_{\ell} P_{\ell} < \frac{1}{\underaccent{\bar}{C}}$ for every $\ell = 1, \ldots, L$.
\end{proof}

\begin{remark}
    In practical terms we can guarantee the condition $A_{\ell} P_{\ell} < \frac{1}{\underaccent{\bar}{C}}$ holds for every $\ell = 1, \ldots, L$ provided $L$ is large enough. This is because $\sum_{\ell = 1}^L A_{\ell} = 1$. So, for fixed $\|PA\|_1$, when $L$ is sufficiently large we can take $A_{\ell}$ to be sufficiently small so that $A_{\ell} P_{\ell} <  \frac{1}{\underaccent{\bar}{C}}$
\end{remark}

\section{Numerical simulations}
\subsection{Consumers}
In the previous analysis, we assumed that consumers play a passive role in the model. In particular, the constant $B$ has been fixed. However, in doing so we have not modelled the active nature of consumers in determining the price index $R$. To include this when implementing our numerical methods we return to~\eqref{eq:consumer-pricing} and we normalise economic output to $Y = 1$. Then, by rearranging~\eqref{eq:consumer-pricing} and using the production function $q_{\ell,i} = Z_{\ell,i}$, we get $B = \left[ \frac{1}{N} \sum_{\ell = 1}^L \sum_{i = 1}^{N_{\ell}} Z_{\ell,i}^{\alpha}\right]^{\frac{1}{\alpha - 1}}$. So, as the number of firms in each sector goes to infinity, $B = \left[ \sum_{\ell = 1}^L A_{\ell} \int_{\Omega} z^{\alpha} m_{\ell}(z)~dz \right]^{\frac{1}{\alpha - 1}}$. Note that this now needs to be solved as a fixed point, as $m_{\ell}$ itself depends on $B$.

\subsection{Simulations}
We computed simulations with synthetic data, using the numerical method outlined in Appendix~\ref{sec:appendix}. From an economics perspective it is important to understand how the model affects the sector--level productivity. The purpose of the simulations is to provide initial insights into the role of the modelling parameters and of the network configuration.

\begin{figure}[b!]
    \begin{subfigure}{0.49 \textwidth}
        \centering
        \includegraphics[width = \linewidth]{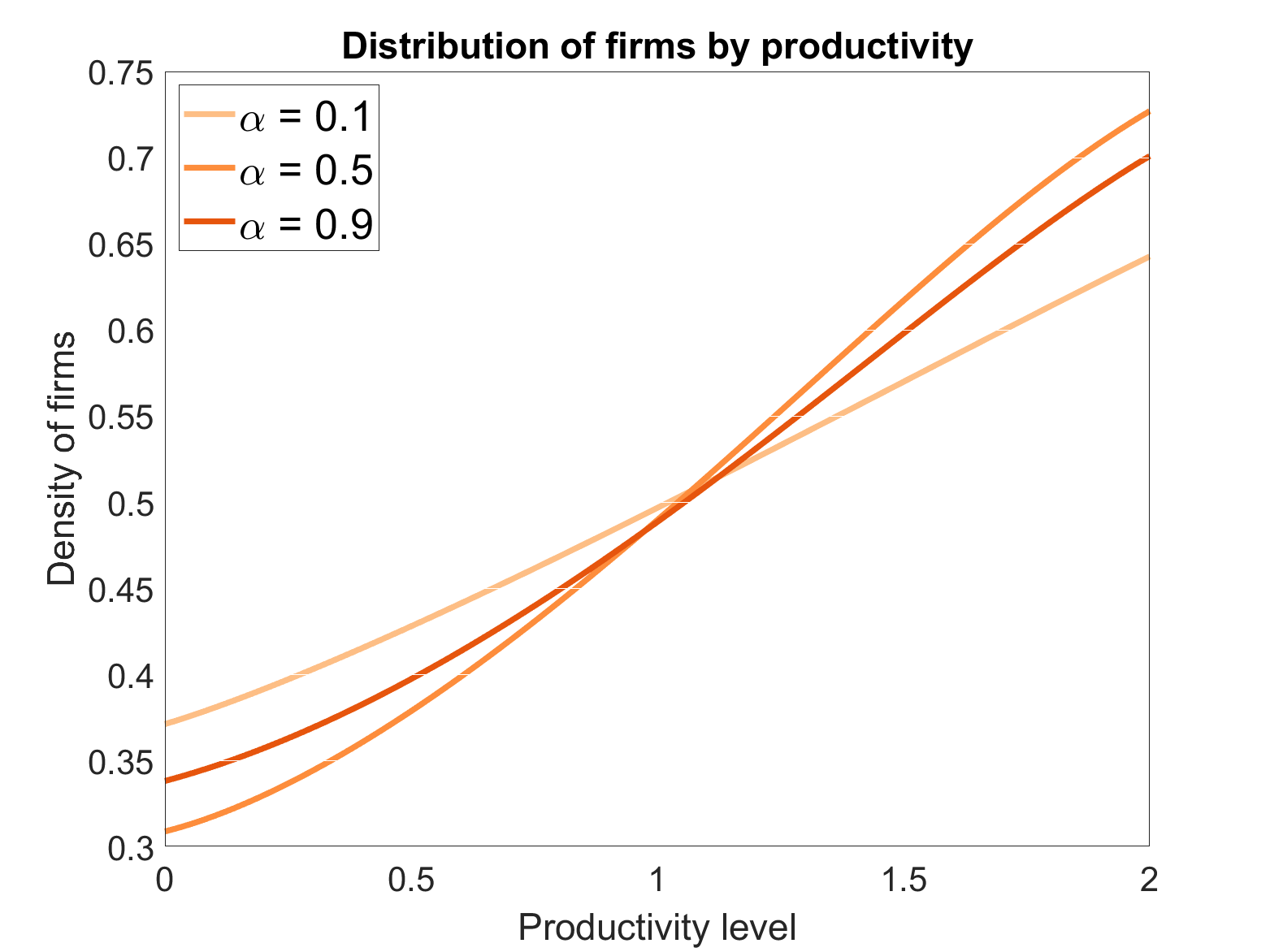}
        \caption{Plot of firm distribution with varying $\alpha$}
        \label{fig:alpha-sim}
    \end{subfigure}
    \begin{subfigure}{0.49 \textwidth}
        \centering
        \includegraphics[width = \linewidth]{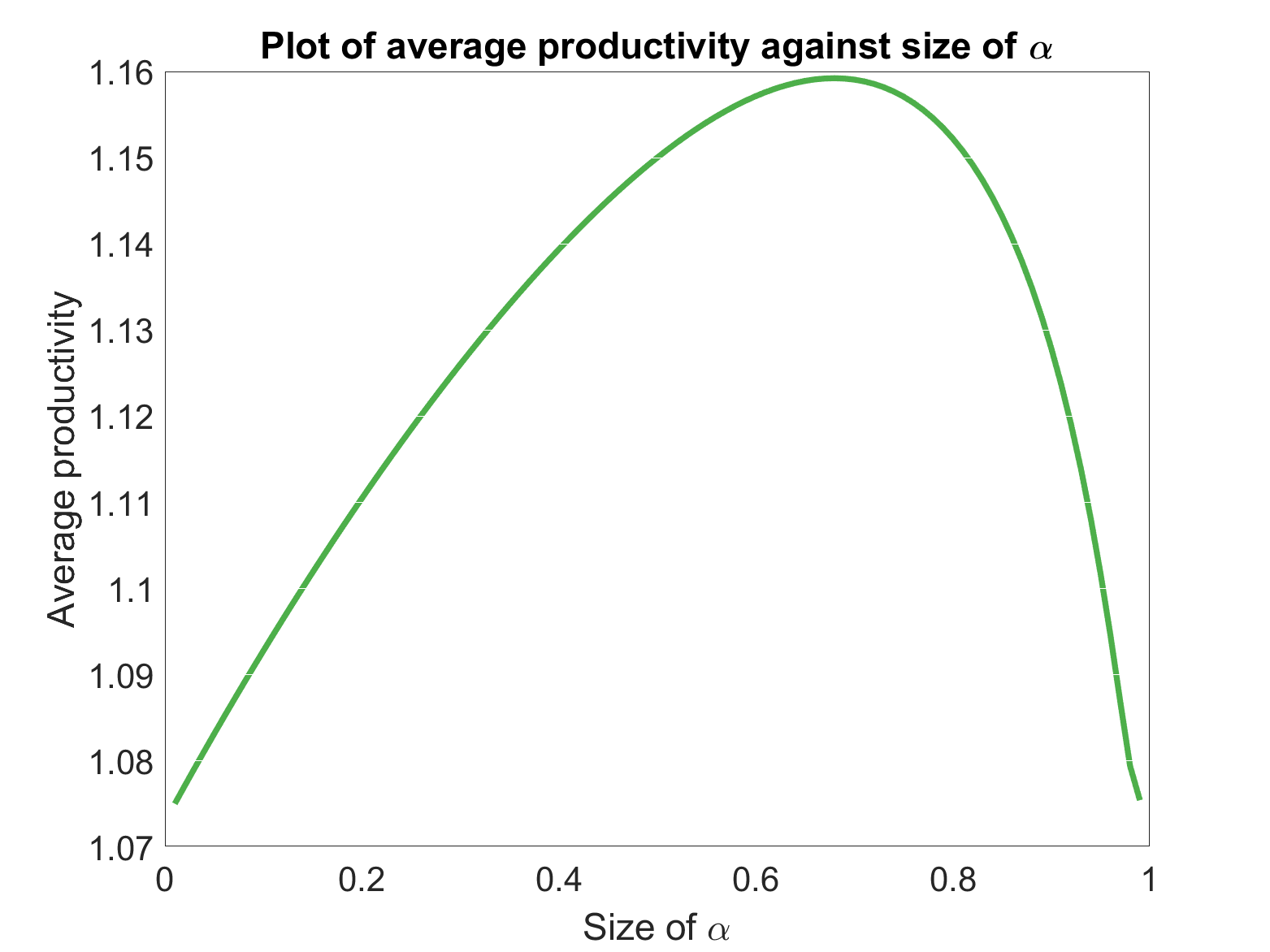}
        \caption{Plot of average productivity against $\alpha$}
        \label{fig:alpha-prod}
    \end{subfigure}
    \begin{subfigure}{0.49 \textwidth}
        \centering
        \includegraphics[width = \linewidth]{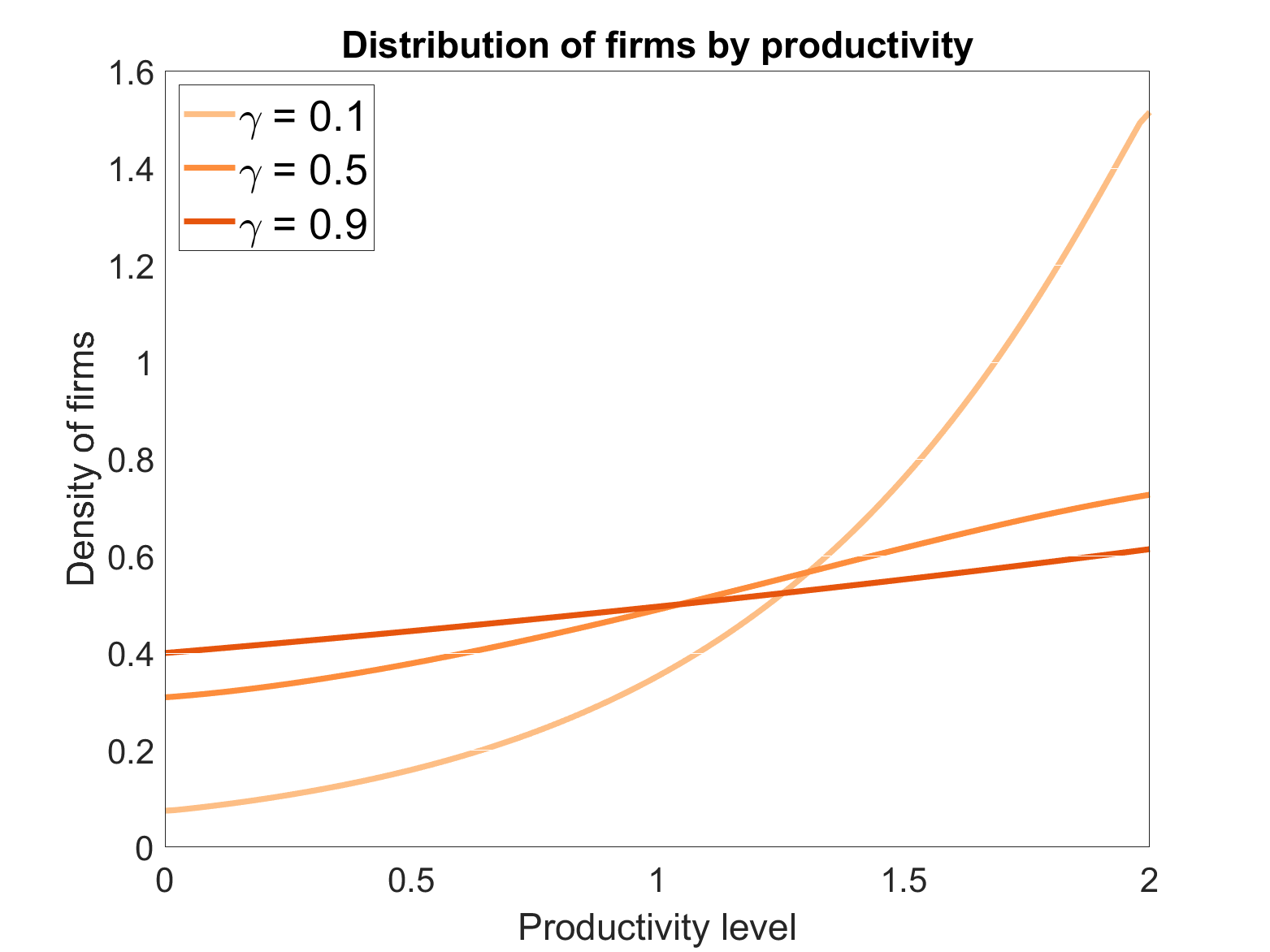}
        \caption{Plot of firm distribution with varying $\gamma$}
        \label{fig:gamma-sim}
    \end{subfigure}
    \begin{subfigure}{0.49 \textwidth}
        \centering
        \includegraphics[width = \linewidth]{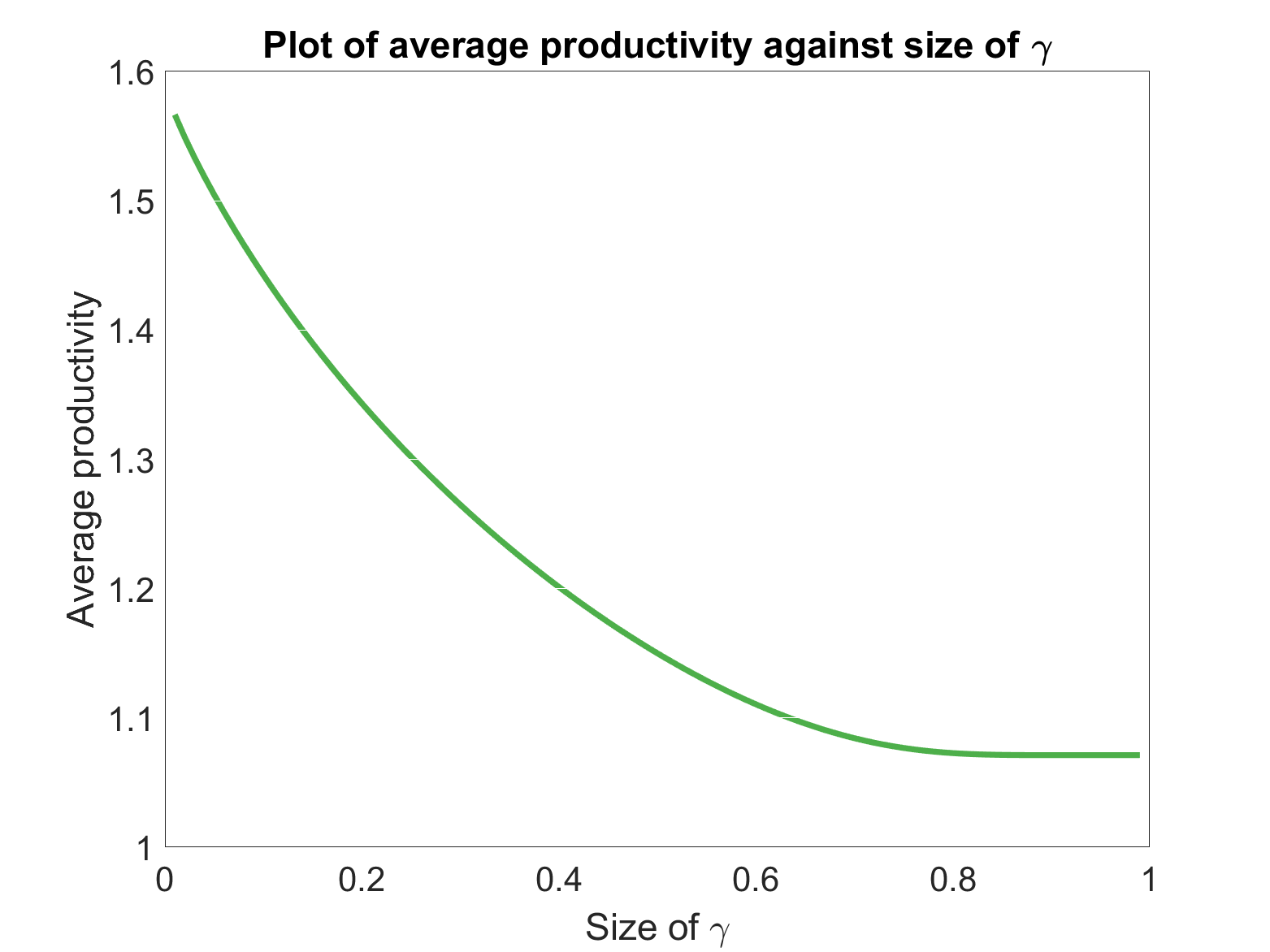}
        \caption{Plot of average productivity against $\gamma$}
        \label{fig:gamma-prod}
    \end{subfigure}
    \begin{subfigure}{0.49 \textwidth}
        \centering
        \includegraphics[width = \linewidth]{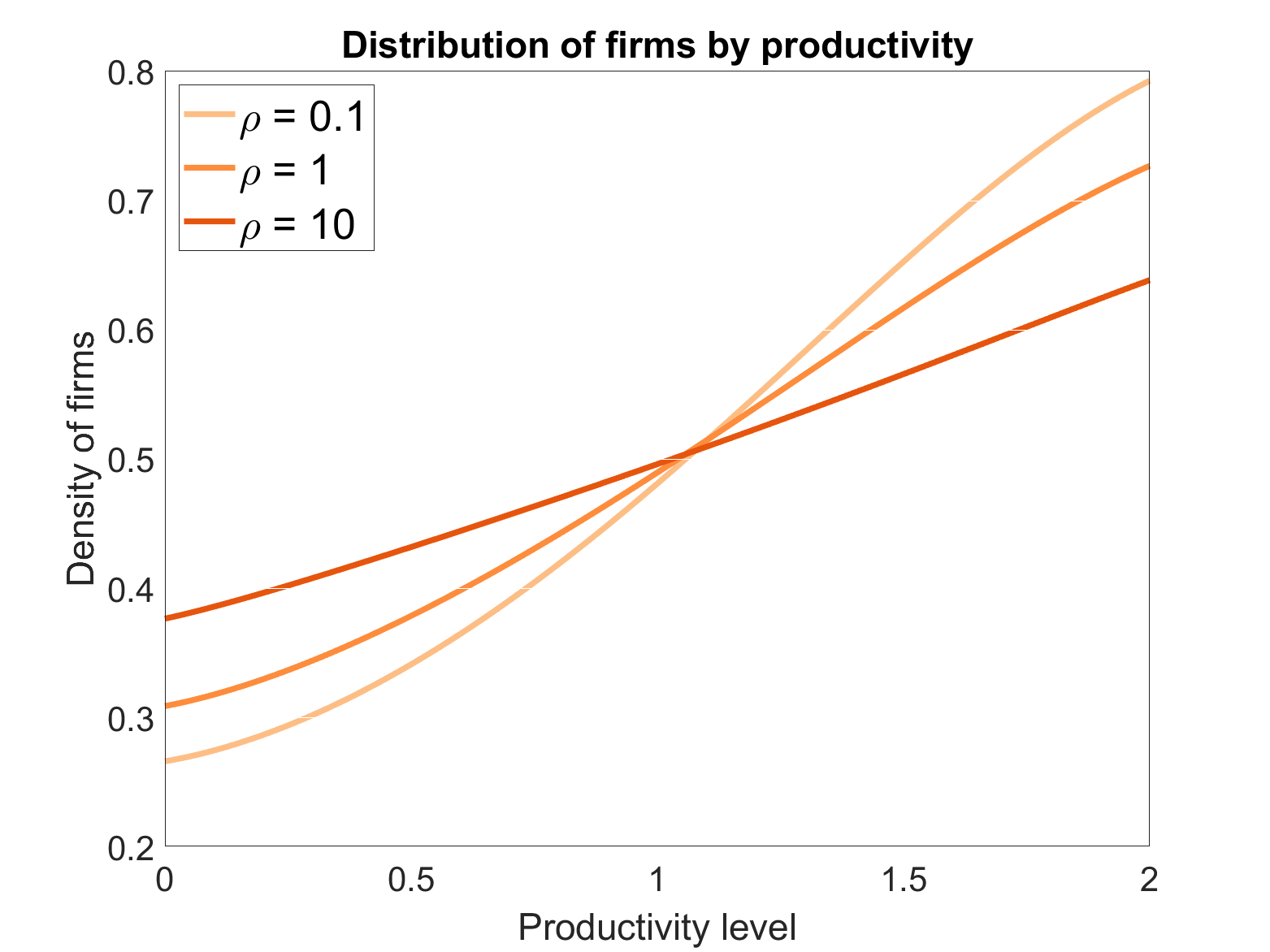}
        \caption{Plot of firm distribution with varying $\rho$}
        \label{fig:rho-sim}
    \end{subfigure}
    \begin{subfigure}{0.49 \textwidth}
        \centering
        \includegraphics[width = \linewidth]{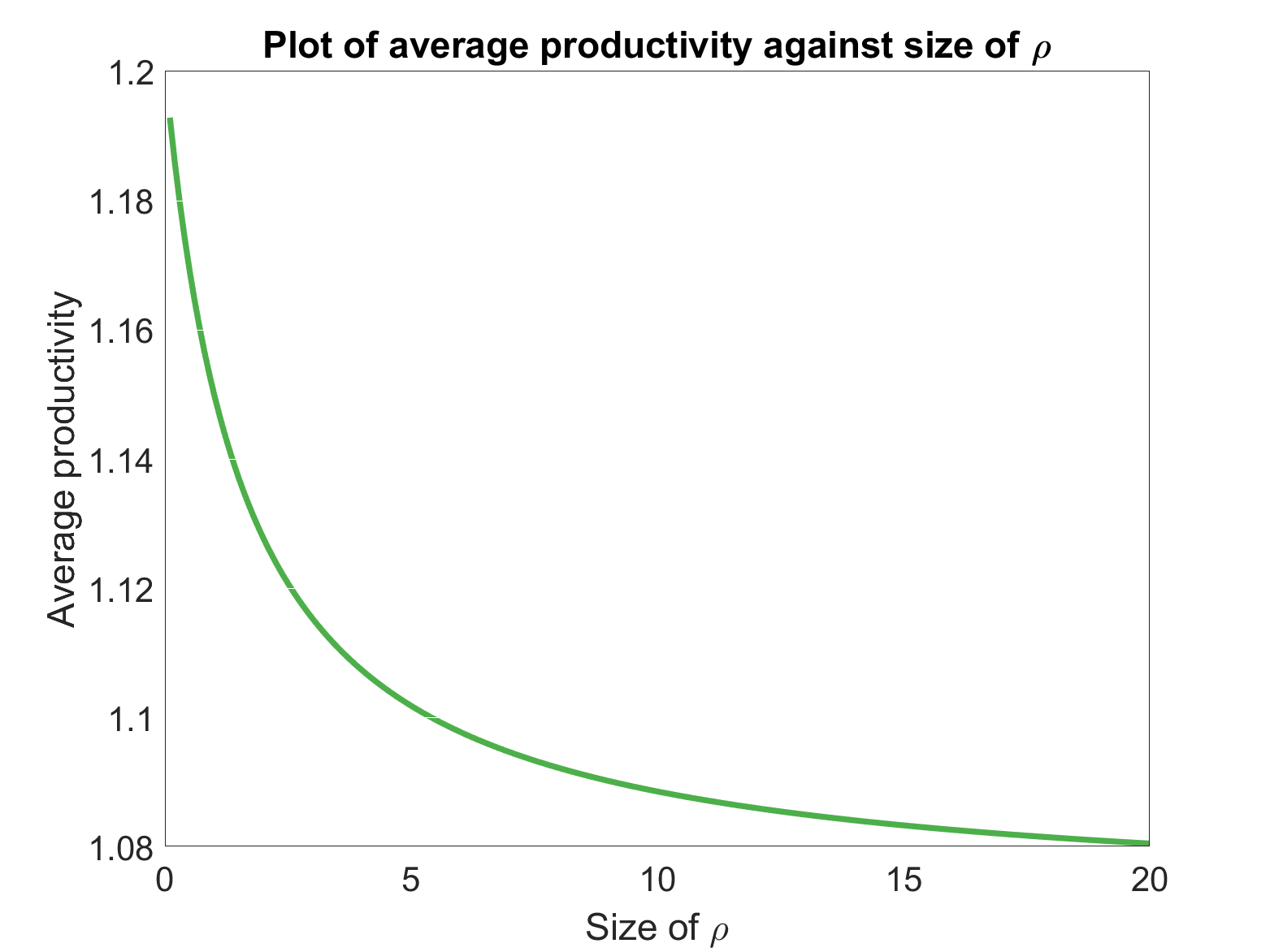}
        \caption{Plot of average productivity against $\rho$}
        \label{fig:rho-prod}
    \end{subfigure}
    \caption{Simulations of MFG with varying $\alpha$, $\gamma$, and $\rho$}
\end{figure}

\subsubsection*{Parameter effects}
The MFG depends on the parameters $\sigma$, $w$, $\alpha$, $\gamma$ and $\rho$. Recall that $\sigma > 0$ is the strength of noise in an individual's dynamics, $w > 0$ is the wage paid to employees, $\alpha \in (0,1)$ is a parameter in the consumer optimisation problem which ensures convexity, and $\gamma \in (0,1)$ is the returns to labour i.e. the inefficiency in converting one unit of labour to one unit of knowledge, it also ensures convexity of the firm--level optimisation problem.

In order to separate the parameter effects from any effects caused by the sector network, we ran simulations with just a single sector. We fixed $\bar{z} = 2$, $A = A_1 = 1$ and $P = 0.1$, where $\bar{z}$ is the maximum productivity level, $A_1$ is the proportion of firms in sector 1 and $P$ is the strength of connection from sector 1 to itself. 
For baseline values, we took $\sigma = 1$, $w = 1$, $\rho = 1$, $\gamma = 0.5$ and $\alpha = 0.5$. For each simulation, we varied one parameter while keeping all others at the baseline level. Figures~\ref{fig:alpha-sim} and~\ref{fig:alpha-prod} show that the relationship between $\alpha$ and the distribution of firms is a complex one. There is some $\alpha^* \in (0, 1)$ where the average productivity reaches a maximum, while on $(0, \alpha^*]$ average productivity is monotonically increasing, and on $[\alpha^*, 1)$ average productivity is monotonically decreasing. Note that, for fixed productivity level and firm distribution, a firm's revenue is $r_{\ell} q_{\ell} = Z_{\ell}^{\alpha} \left[\frac{1}{N} \sum_{\ell' = 1}^L A_{\ell'} \int_{\Omega} z^{\alpha} m_{\ell'}(z)~dz\right]^{-1}$, which consists of a term that increases with respect to $\alpha$ multiplied by a term that decreases with respect to $\alpha$. This results in a competing effect between $\alpha$ and a firm's revenue, which in turn affects a firm's return on investment, and therefore its level of investment in labour. Since labour investment has an increasing effect on average productivity, the competing terms in the revenue equation directly correspond to the behaviour exhibited in figure \ref{fig:alpha-prod}.

Figures~\ref{fig:gamma-sim} and~\ref{fig:gamma-prod} shows the effect of $\gamma$ on the sector--level productivity. Figure~\ref{fig:gamma-prod} shows that as $\gamma$ increases, the average productivity decreases. Since $\gamma$ relates to the inefficiency of converting one unit of labour to one unit of productive work, it seems counter--intuitive at first that average productivity would be a decreasing function of $\gamma$. Recall that the optimal level of employment is given by $h^* = \left(\frac{\gamma}{w} \max(0,V')\right)^{\frac{1}{1 - \gamma}}$, which increases productivity at a rate $(h^*)^{\gamma}$. Then, $h^*$ is increasing with respect to $\gamma$ for fixed $V'$ if and only if $V' \geq \frac{w}{\gamma} e^{\frac{\gamma - 1}{\gamma}}$ and $(h^*)^{\gamma}$ is increasing if and only if $V' \geq \frac{w}{\gamma} e^{\gamma - 1}$. Hence, the effect of $\gamma$ on the average productivity depends on $V'$ and how it changes with respect to $\gamma$.

The effects of $\rho$ and $\sigma$ on the average productivity, shown in Figures~\ref{fig:rho-sim},~\ref{fig:rho-prod} and~\ref{fig:sigma-sim},~\ref{fig:sigma-prod} respectively, show the same trend: average productivity decreases as each parameter increases. The size of $\rho$ is the extent to which a firm discounts future profits. As $\rho$ increases, firms care less about the future state of the system and so they are less willing to invest in labour; it is an investment whose effect is only on the future value of productivity. This results in reduced average productivity in the long run, which can be seen in Figure~\ref{fig:rho-prod}. As $\sigma$ increases, the randomness in productivity evolution of each firm increases. So, the impact of labour on productivity decreases with increasing $\sigma$, and this is reflected in Figure~\ref{fig:sigma-prod}.

Finally, Figures~\ref{fig:w-sim} and~\ref{fig:w-prod} shows that average productivity also decreases with increasing wage, $w$. The wage rate increases the cost of labour. So, we can directly see that as the wage increases, the optimal level of employment, and hence the average productivity, decreases.
    
\begin{figure}[t]
    \begin{subfigure}{0.49 \textwidth}
        \centering
        \includegraphics[width = \linewidth]{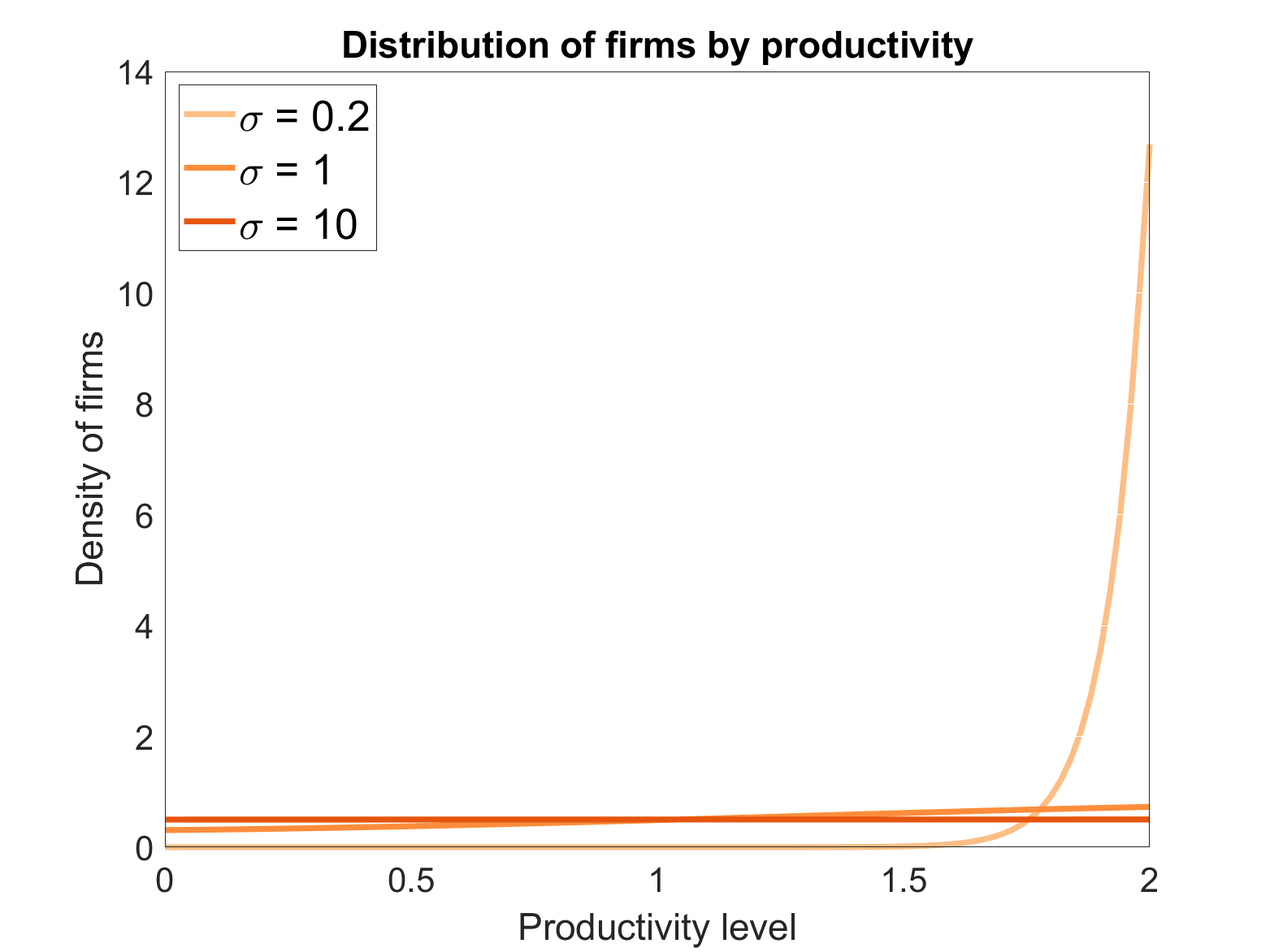}
        \caption{Plot of firm distribution with varying $\sigma$}
        \label{fig:sigma-sim}
    \end{subfigure}
    \begin{subfigure}{0.49 \textwidth}
        \centering
        \includegraphics[width = \linewidth]{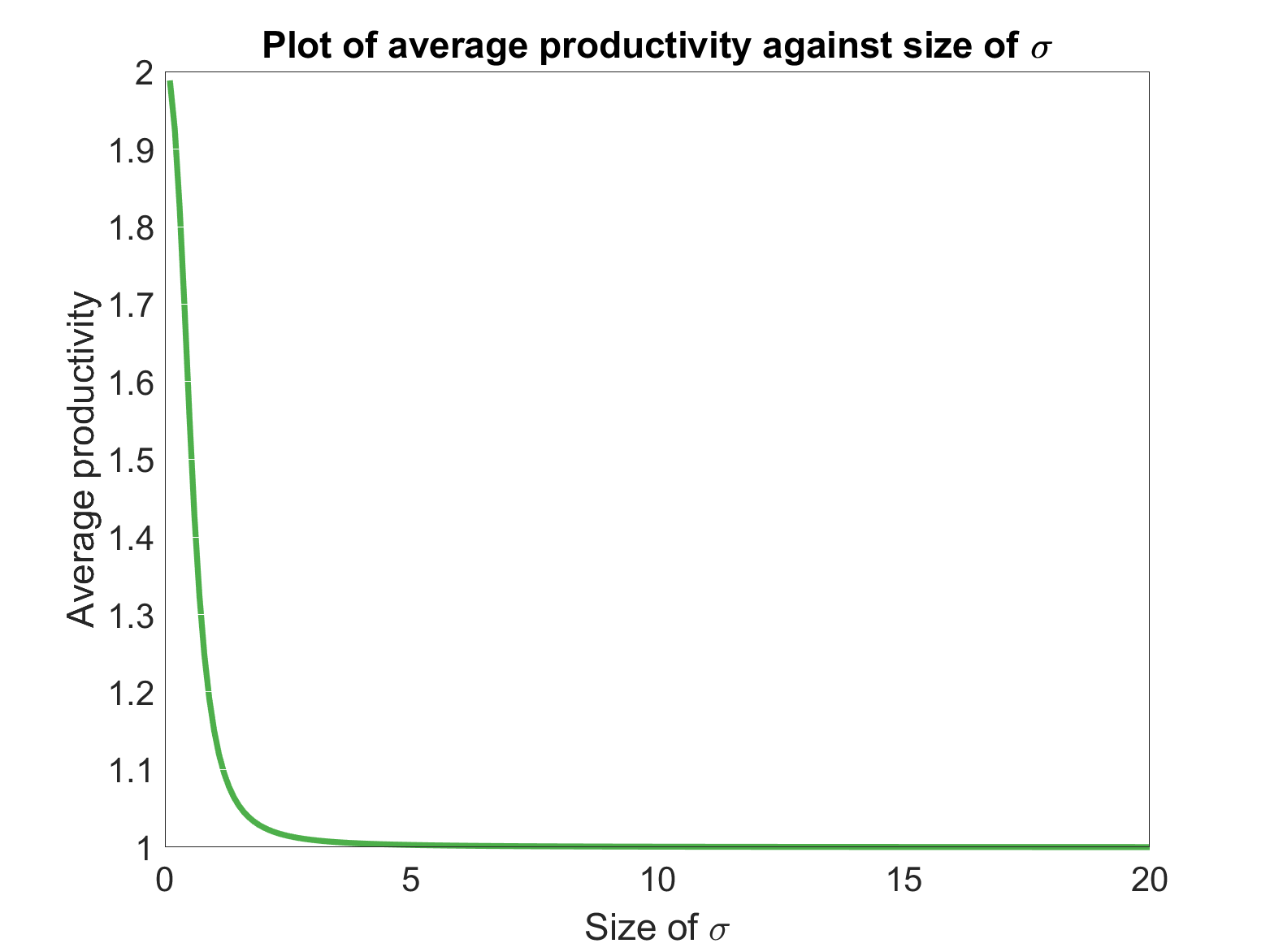}
        \caption{Plot of average productivity against $\sigma$}
        \label{fig:sigma-prod}
    \end{subfigure}
    \begin{subfigure}{0.49 \textwidth}
        \centering
        \includegraphics[width = \linewidth]{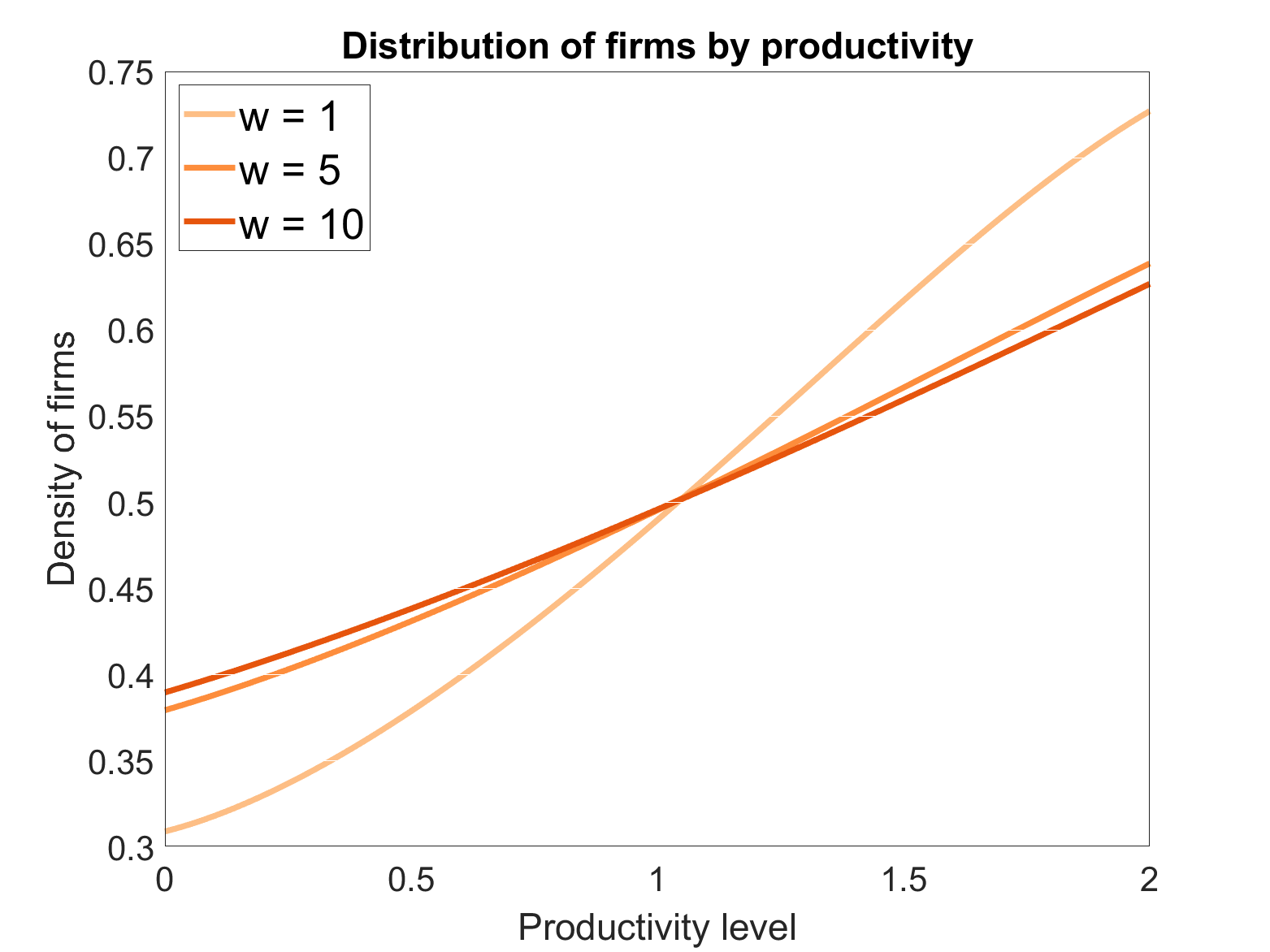}
        \caption{Plot of firm density against productivity with $w = 1,5,10$}
        \label{fig:w-sim}
    \end{subfigure}
    \begin{subfigure}{0.49 \textwidth}
        \centering
        \includegraphics[width = \linewidth]{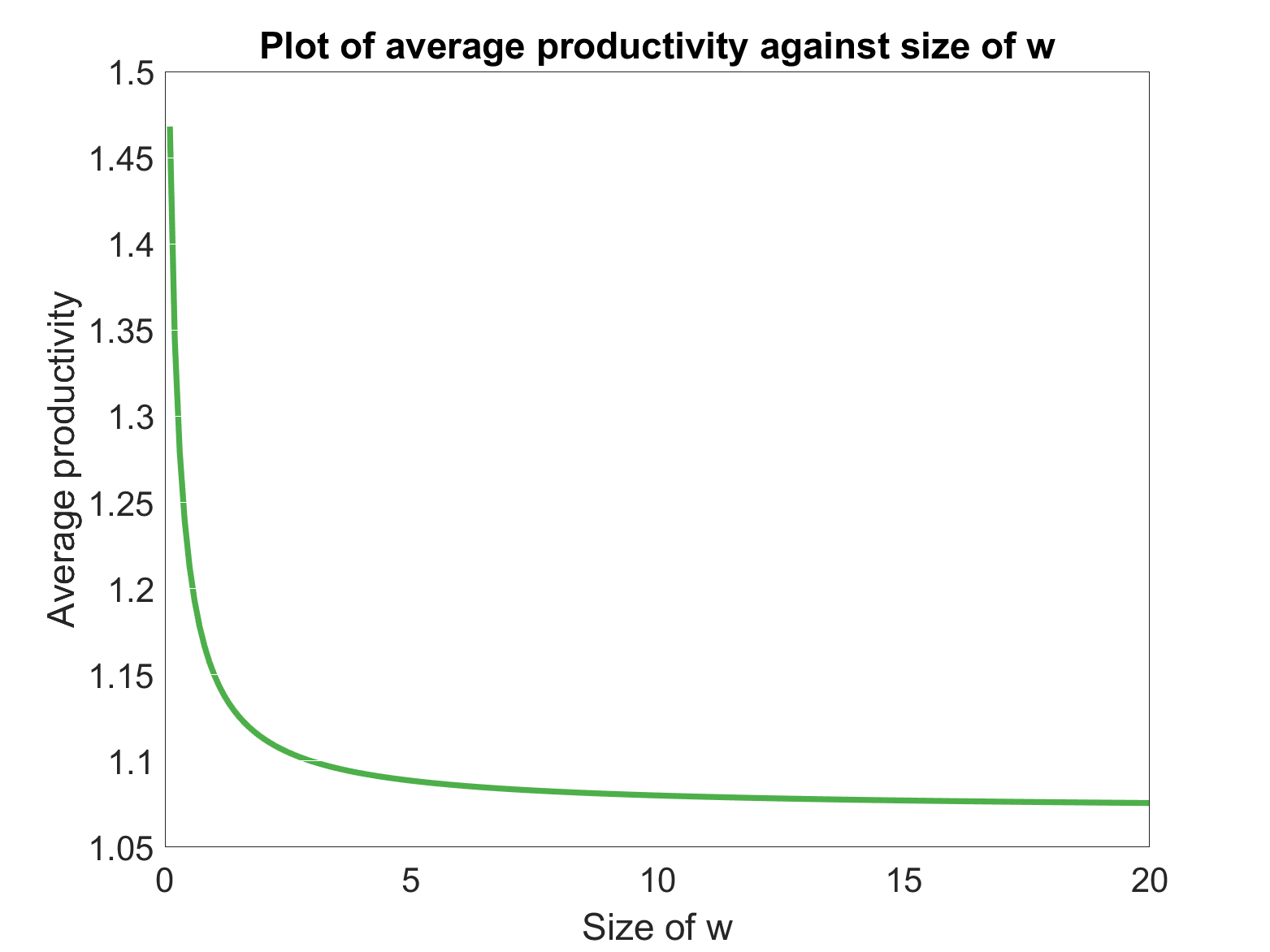}
        \caption{Plot of average productivity against value of $w$}
        \label{fig:w-prod}
    \end{subfigure}
    \caption{Simulations of MFG with varying $\sigma$ and $w$}
     \label{fig:w}
\end{figure}

\subsubsection*{Spillover size effects}
The sector--level network, encoded by the vertex weights $A_{\ell}$ for sector $\ell$, and the edge weights $p(\ell,\ell')$ for a transfer of knowledge from sector $\ell'$ to sector $\ell$, is called the spillover network as it describes how knowledge and productivity spills over from one sector to another. A path in the spillover network is called a spillover path, or just spillover if there is no ambiguity. A path of length 1 from sector $\ell'$ to sector $\ell$ is called a direct spillover, a path of length 2 or greater from sector $\ell'$ to sector $\ell$ is called an indirect spillover, and in both cases sector $\ell$ is called the receiving sector and sector $\ell'$ is the originating sector.

In almost all economic literature, only direct spillovers have been modelled and we are aware of no models that pay attention to the effect indirect spillovers have on economic productivity. In this subsection, we begin investigating how the productivity of a sector is affected by the structure of the spillover network, and in particular the effect of indirect spillovers on productivity. To undertake this investigation, we conducted three types of simulations. The first simulations were to model the six networks in Figure~\ref{fig:networks}, to provide initial insight into how indirect spillover paths  affect the distribution of firms. In the second simulations, we randomly generated spillover networks in models with three sectors and used the collected data to hypothesise a relationship between the average productivity of a sector and the size of spillovers (direct and indirect) it received. In the final simulations, we tested our hypothesis on more randomly generated spillover networks, this time for models with 10 sectors, which more closely resembles the number of sectors in the real economy. We showed that the hypothesis developed accurately describes the relationship between the spillover network and the average productivity of firms, moreover there was a $20\%$ reduction in error when direct and indirect spillovers were taken into account, compared with when only direct spillovers were considered. Therefore, our conclusion from this preliminary investigation is that indirect spillovers have a significant effect on economic productivity in our model and they should not be ignored.

The networks in Figure~\ref{fig:networks} provide insight into how indirect spillovers affect the distribution of firms, in comparison to direct spillovers. In network 1, sector $C$ has one direct spillover, in network 2 it has one direct spillover and one indirect spillover of length 2, and in network 3 it has two direct spillovers. So, the difference in productivity in sector $C$ between network 2 and network 1 will show the effect of an indirect spillover compared with having no spillover, and the difference between networks 3 and 2 will show the effect of an indirect spillover compared with a direct spillover. The differences in density of sector C are plotted in Figure~\ref{fig:network-123-diff-sim}. From the plots, it can be seen that the density of firms is larger at high productivity levels in network 2 compared with network 1 and the density is lower at low productivity levels. This means that the indirect spillover from sector $A$ to sector $C$ has a positive effect on sector C, skewing the distribution towards higher productivity levels. The same behaviour can be seen when we compare sector C in network 3 to network 2, however the effect is an order of magnitude larger. Therefore, although an indirect spillover path has some positive effect compared with no path at all, the effect is less strong than a direct spillover path.

In Figure~\ref{fig:network-456-diff-sim}, sector $D$ of networks four to six were modelled. For sector D: in network 4 there is one indirect spillover with path length 2; network 5 has one indirect spillover with path length 2 and one with path length 3; finally network 6 has an infinite number of indirect spillovers, one for every path length. We have plotted the difference in density of sector D between network 5 and network 4 and between networks 6 and 5. The difference between network 5 and network 4 shows the effect of an indirect spillover of length 3, while the difference between network 6 and network 5 shows the effect of indirect spillovers of all lengths greater than 3. For the difference between network 5 and network 4, the same qualitative result as the difference between network 2 and network 1, in Figure~\ref{fig:network-123-diff-sim}, is observed. This suggests that having spillover paths of greater length do have positive impacts on productivity, but with reduced impact for increased path lengths. Interestingly, sector D in network 6 is less productive than sector D in sector 5. Further investigation showed that if $B$ is fixed, rather than the solution of a fixed point problem, then the effect that more paths result in greater productivity returns, see figure~\ref{fig:network-456-diff-sim-alt}. The reason for this is not immediately obvious and warrants further study. Since the observed change is very small, it can't be ruled out that this result is an artefact from simplifications in the model.

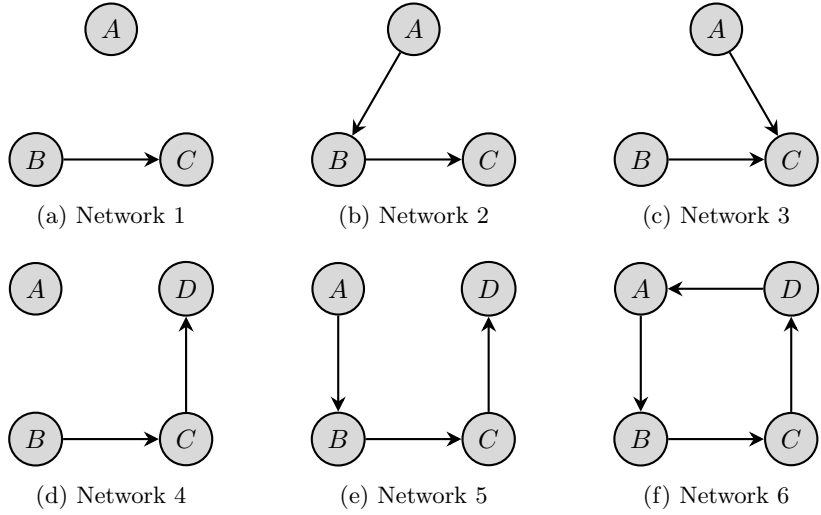
\begin{figure}[t!]
    \centering
    \begin{subfigure}{0.3 \textwidth}
        \centering
        \begin{tikzpicture} [scale = 0.5]
            \node (a) at (2,3.464) [shape = circle, draw = black, fill = gray, fill opacity = 0.3, text opacity = 1, thick] {$A$};
            \node (b) at (0,0) [shape = circle, draw = black, fill = gray, fill opacity = 0.3, text opacity = 1, thick] {$B$};
            \node (c) at (4,0) [shape = circle, draw = black, fill = gray, fill opacity = 0.3, text opacity = 1, thick] {$C$};
            \draw [->, -{Stealth [length = 2mm, width = 2mm]}, thick] (b) -- (c);
        \end{tikzpicture}
        \caption{Network 1}
    \end{subfigure}
    \begin{subfigure}{0.3 \textwidth}
        \centering
        \begin{tikzpicture} [scale = 0.5]
           \node (a) at (2,3.464) [shape = circle, draw = black, fill = gray, fill opacity = 0.3, text opacity = 1, thick] {$A$};
            \node (b) at (0,0) [shape = circle, draw = black, fill = gray, fill opacity = 0.3, text opacity = 1, thick] {$B$};
            \node (c) at (4,0) [shape = circle, draw = black, fill = gray, fill opacity = 0.3, text opacity = 1, thick] {$C$};
            \draw [->, -{Stealth [length = 2mm, width = 2mm]}, thick] (a) -- (b);
            \draw [->, -{Stealth [length = 2mm, width = 2mm]}, thick] (b) -- (c);
        \end{tikzpicture}
        \caption{Network 2}
    \end{subfigure}
    \begin{subfigure}{0.3 \textwidth}
        \centering
        \begin{tikzpicture} [scale = 0.5]
            \node (a) at (2,3.464) [shape = circle, draw = black, fill = gray, fill opacity = 0.3, text opacity = 1, thick] {$A$};
            \node (b) at (0,0) [shape = circle, draw = black, fill = gray, fill opacity = 0.3, text opacity = 1, thick] {$B$};
            \node (c) at (4,0) [shape = circle, draw = black, fill = gray, fill opacity = 0.3, text opacity = 1, thick] {$C$};
            \draw [->, -{Stealth [length = 2mm, width = 2mm]}, thick] (a) -- (c); 
            \draw [->, -{Stealth [length = 2mm, width = 2mm]}, thick] (b) -- (c);
        \end{tikzpicture}
        \caption{Network 3}
    \end{subfigure}
    \begin{subfigure}{0.3 \textwidth}
        \centering
        \begin{tikzpicture} [scale = 0.5]
            \node (a) at (0,4) [shape = circle, draw = black, fill = gray, fill opacity = 0.3, text opacity = 1, thick] {$A$};
            \node (b) at (0,0) [shape = circle, draw = black, fill = gray, fill opacity = 0.3, text opacity = 1, thick] {$B$};
            \node (c) at (4,0) [shape = circle, draw = black, fill = gray, fill opacity = 0.3, text opacity = 1, thick] {$C$};
            \node (d) at (4,4) [shape = circle, draw = black, fill = gray, fill opacity = 0.3, text opacity = 1, thick] {$D$};
            \draw [->, -{Stealth [length = 2mm, width = 2mm]}, thick] (b) -- (c);
            \draw [->, -{Stealth [length = 2mm, width = 2mm]}, thick] (c) -- (d);
        \end{tikzpicture}
        \caption{Network 4}
    \end{subfigure}
    \begin{subfigure}{0.3 \textwidth}
        \centering
        \begin{tikzpicture} [scale = 0.5]
           \node (a) at (0,4) [shape = circle, draw = black, fill = gray, fill opacity = 0.3, text opacity = 1, thick] {$A$};
            \node (b) at (0,0) [shape = circle, draw = black, fill = gray, fill opacity = 0.3, text opacity = 1, thick] {$B$};
            \node (c) at (4,0) [shape = circle, draw = black, fill = gray, fill opacity = 0.3, text opacity = 1, thick] {$C$};
            \node (d) at (4,4) [shape = circle, draw = black, fill = gray, fill opacity = 0.3, text opacity = 1, thick] {$D$};
            \draw [->, -{Stealth [length = 2mm, width = 2mm]}, thick] (a) -- (b);
            \draw [->, -{Stealth [length = 2mm, width = 2mm]}, thick] (b) -- (c);
            \draw [->, -{Stealth [length = 2mm, width = 2mm]}, thick] (c) -- (d);
        \end{tikzpicture}
        \caption{Network 5}
    \end{subfigure}
    \begin{subfigure}{0.3 \textwidth}
        \centering
        \begin{tikzpicture} [scale = 0.5]
            \node (a) at (0,4) [shape = circle, draw = black, fill = gray, fill opacity = 0.3, text opacity = 1, thick] {$A$};
            \node (b) at (0,0) [shape = circle, draw = black, fill = gray, fill opacity = 0.3, text opacity = 1, thick] {$B$};
            \node (c) at (4,0) [shape = circle, draw = black, fill = gray, fill opacity = 0.3, text opacity = 1, thick] {$C$};
            \node (d) at (4,4) [shape = circle, draw = black, fill = gray, fill opacity = 0.3, text opacity = 1, thick] {$D$};
            \draw [->, -{Stealth [length = 2mm, width = 2mm]}, thick] (a) -- (b); 
            \draw [->, -{Stealth [length = 2mm, width = 2mm]}, thick] (b) -- (c);
            \draw [->, -{Stealth [length = 2mm, width = 2mm]}, thick] (c) -- (d);
            \draw [->, -{Stealth [length = 2mm, width = 2mm]}, thick] (d) -- (a);
        \end{tikzpicture}
        \caption{Network 6}
    \end{subfigure}
    \caption{Sector--level networks for simulations in Figures~\ref{fig:network-123-diff-sim} and~\ref{fig:network-456-diff-sim}}
    \label{fig:networks}
\end{figure}

\begin{figure}[t!]
    \centering
    \begin{subfigure}{0.49 \textwidth}
        \centering
        \includegraphics[width = \linewidth]{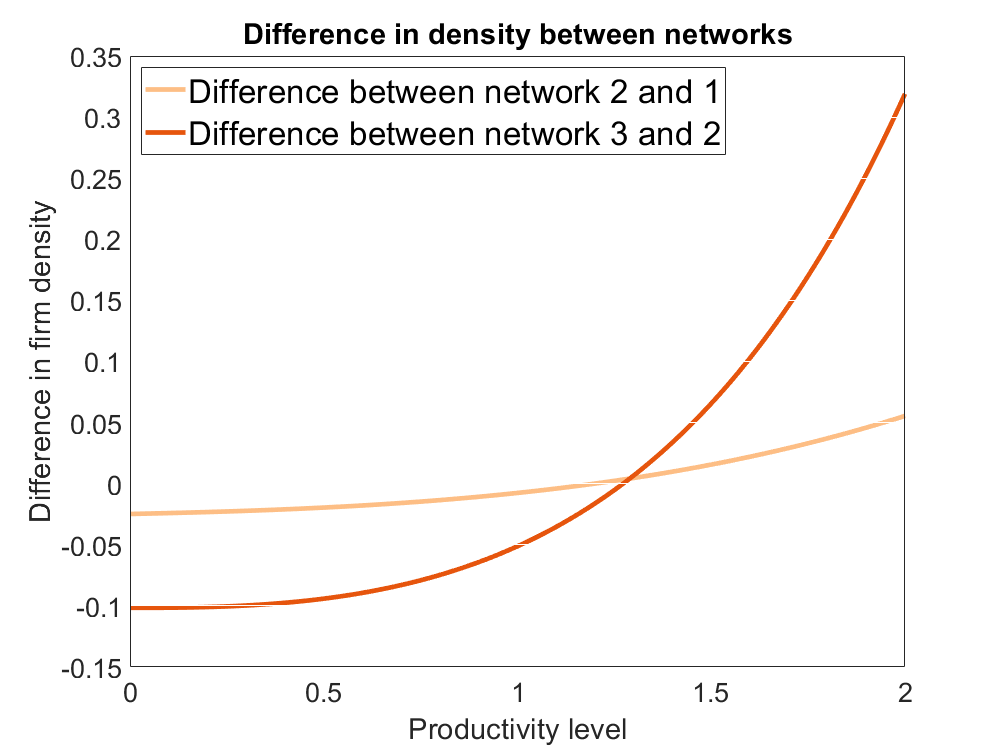}
        \caption{Difference in density of sector C between networks 2 and 1, and networks 3 and 2}
        \label{fig:network-123-diff-sim}
    \end{subfigure}
    \begin{subfigure}{0.49 \textwidth}
        \centering
        \includegraphics[width = \linewidth]{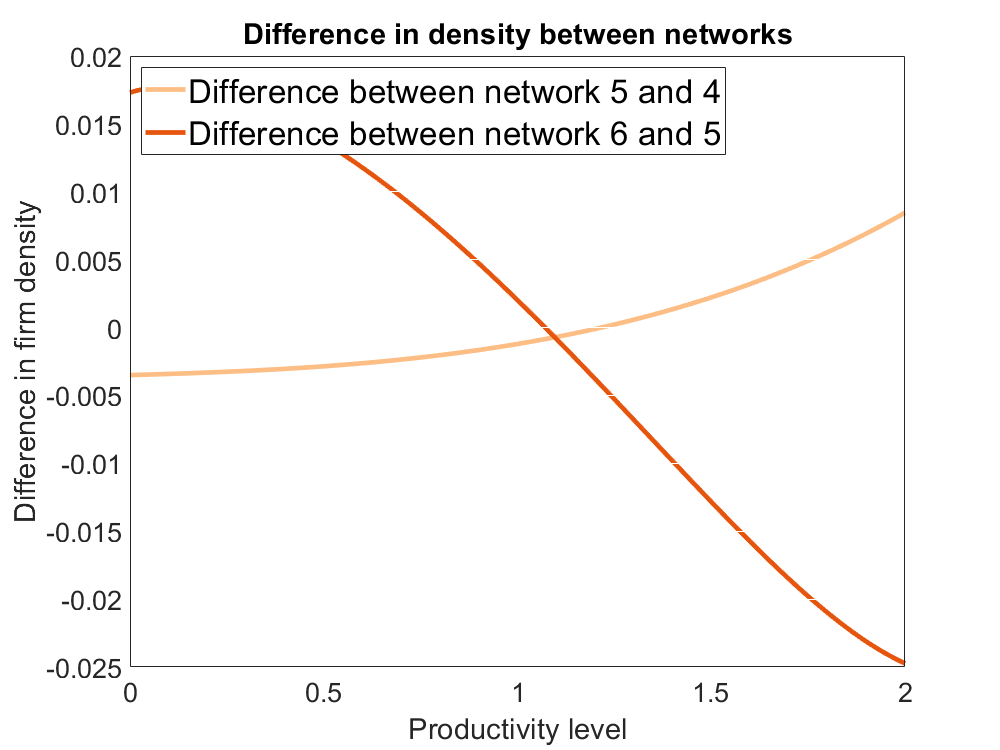}
        \caption{Difference in density of sector D between networks 5 and 4, and networks 6 and 5}
        \label{fig:network-456-diff-sim}
    \end{subfigure}
    \begin{subfigure}{0.49 \textwidth}
        \centering
        \includegraphics[width = \linewidth]{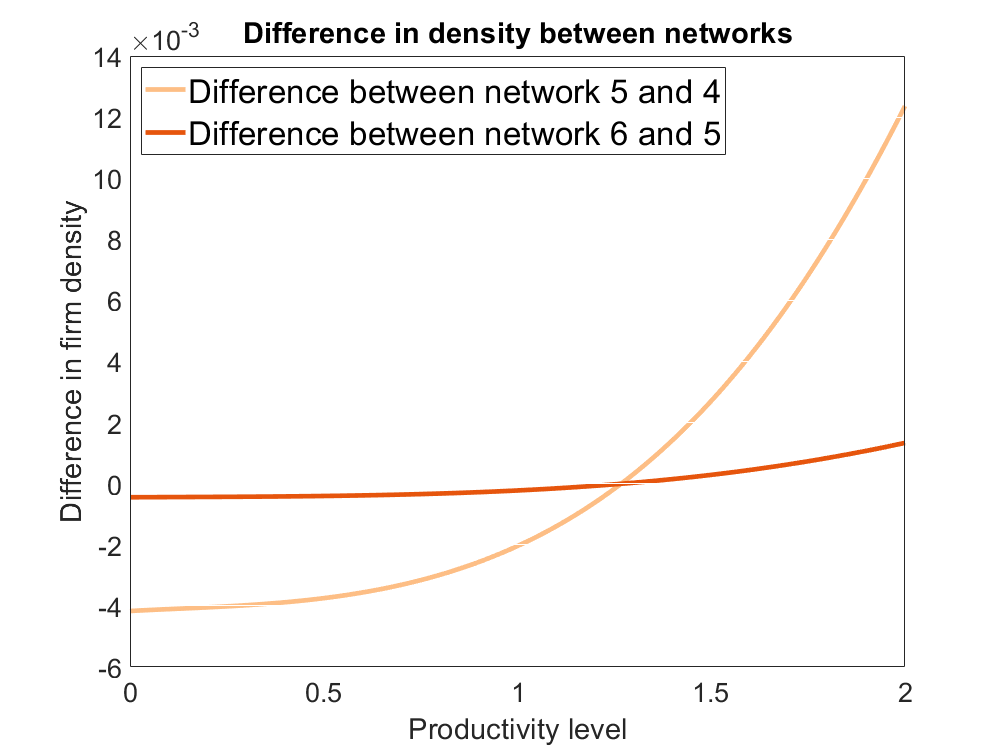}
        \caption{Difference in density of sector D between networks 5 and 4, and networks 6 and 5, with fixed $B = 1$}
        \label{fig:network-456-diff-sim-alt}
    \end{subfigure}
    \caption{Simulations of MFG comparing distribution of firms in sectors C and D with respect to productivity in networks 1, 2 and 3, and networks 4, 5 and 6 from Figure~\ref{fig:networks}}
     \label{fig:network-diff-sim}
\end{figure}

In the second set of simulations, we took a closer look at how the spillover network structure affects the average productivity within each sector. Recall that if, given a network, we know the value of the fixed point, $k^*$, of the function $\Phi$ defined in Definition \ref{def:mfg-phi}. Then the average productivity in sector $\ell$ is $\int_{\Omega} z m_{\ell}(z)~dz = \int_{\Omega} z m^{k^*_{\ell}}(z)~dz$. So, to understand the relationship between average productivity and the network, we first need to understand the relationship between $\int_{\Omega} z m^{k_{\ell}}(z)~dz$ and $k_{\ell}$, for any $k_{\ell} \geq 0$. Then, we also need to understand the relationship between $k^*_{\ell}$ and the $L \times L$ matrix $S$ with entries defined by $S_{\ell,\ell'} = A_{\ell'} p(\ell,\ell')$, because
\[
	k^* = S \left( \int_{\Omega} z m^{k^*_{\ell}}(z)~dz \right)_{\ell = 1}^L \, .
\]
In Figure~\ref{fig:k-sim}, we have plotted $\int_{\Omega} z m^{k}(z)~dz$ against $k$. The relationship appears to approximately follow
\begin{equation} \label{eq:avg-relation}
    \int_{\Omega} z m^{k}(z)~dz = \bar{z} - \frac{b_0}{k^{b_1} + b_2} \, ,
\end{equation}
for some $b_0,b_1,b_2 > 0$, as can be seen by the second line in Figure~\ref{fig:k-sim}. To understand the relationship between the fixed point of $\Phi$ and the matrix $S$, we considered networks of three vertices, with $A_{\ell} = 1/3$ for all $\ell$. We created a random network between the vertices by choosing a connection probability $p$, and making a directed edge between vertices with probability $p$. We then weighted each directed edge with a random weight, chosen from a uniform distribution on $[0,1]$. We repeated this 100 times for each connection probability, and recorded both the size of direct spillovers to each sector and the value of the fixed point of $\Phi$. Figures~\ref{fig:scat-plot-1} and~\ref{fig:scat-plot-2} shows a scatter plot of $k^*_{\ell}$, the $\ell^{th}$ co--ordinate of the fixed point of $\Phi$, against the sum of direct spillover strengths $\sum_{\ell' = 1}^L S_{\ell,\ell'}$. In the simulations with a high connection probability, Figure~\ref{fig:scat-plot-1}, there is a strong linear relationship between $k^*_{\ell}$ and $\sum_{\ell' = 1}^L S_{\ell,\ell'}$. However, with low connection probabilities, Figure~\ref{fig:scat-plot-2}, the simulations tend to follow one of two weaker linear relationships with the row sum.

To understand the relationships further, we can look at the equation that $k^*_{\ell} \in [0,\infty)$ implicitly satisfies: $k^*_{\ell} = \sum_{\ell' = 1}^L S_{\ell,\ell'} \int_{\Omega} z m^{k^*_{\ell'}}~dz$, where $m^{k_{\ell}}$ is defined by~\eqref{eq:mfg-alt-FPsol}. So, if sector $\ell$ receives no spillovers then $k^*_{\ell} = 0$. If it has only direct spillovers, then it is only connected to sectors with no spillovers. So, by defining $f(k) = \int_{\Omega} z m^k~dz$
\begin{equation} \label{eq:k-relation-1}
    k^*_{\ell} = f(0) \sum_{\ell' = 1}^L S_{\ell,\ell'} \, .
\end{equation}
We can see this linear relationship between $k^*_{\ell}$ and $\sum_{\ell' = 1}^L S_{\ell,\ell'}$ in Figures~\ref{fig:scat-plot-3} and~\ref{fig:scat-plot-4}, where we have taken the simulated points in Figure~\ref{fig:scat-plot-2}, and split the data into those points which have only direct spillovers and those that have indirect spillovers as well. In Figure~\ref{fig:scat-plot-3}, where sectors with only direct spillover paths are considered, the linear relationship described by~\eqref{eq:k-relation-1} can be clearly seen.
\begin{figure}[t!]
    \centering
    \includegraphics[width = 0.5 \linewidth]{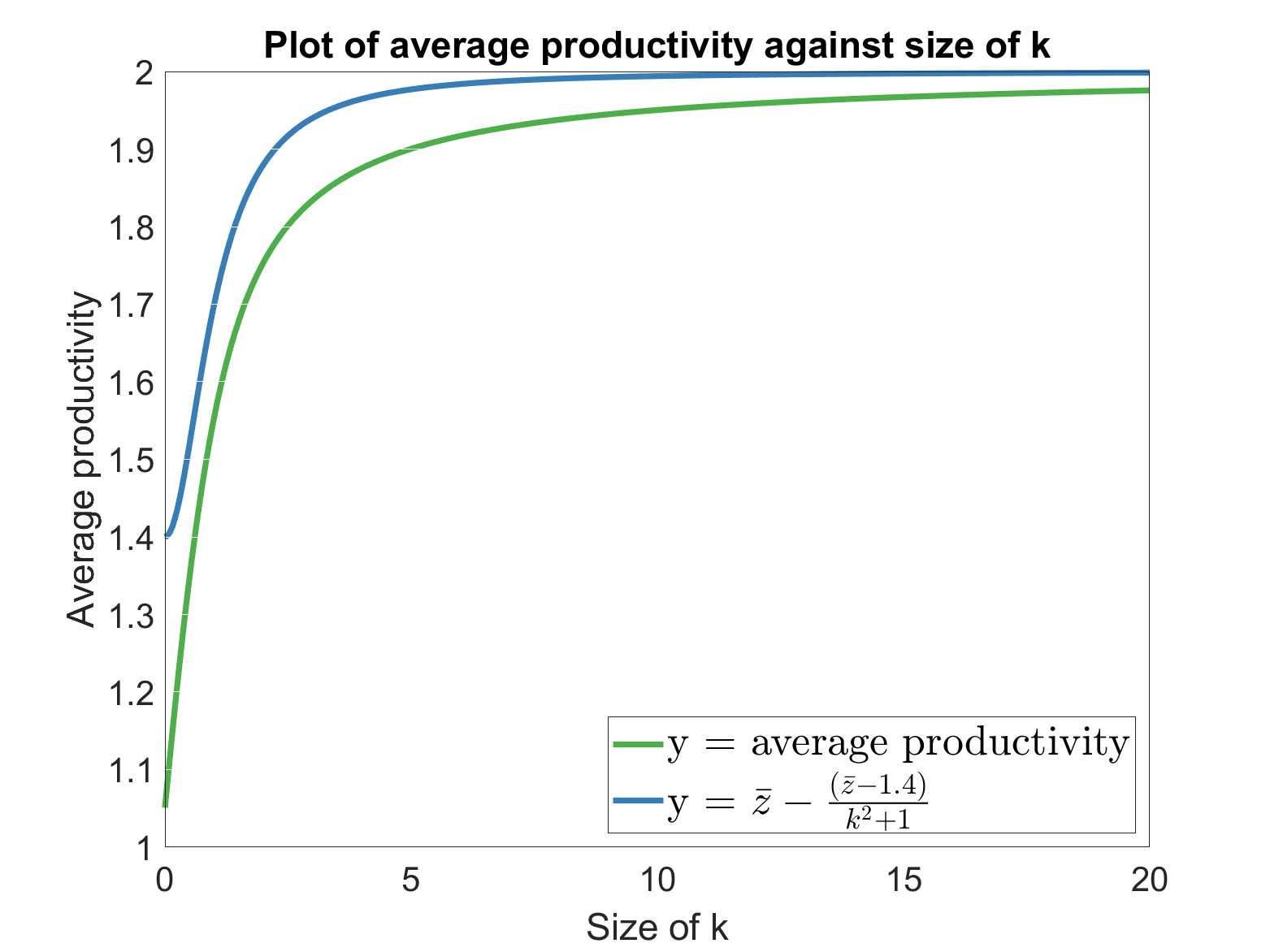}
    \caption{Plot of average productivity against size of $k$ in auxiliary Fokker--Planck equation~\eqref{eq:mfg-alt-model-FP} and plot of $y = \bar{z} - \frac{(\bar{z} - 1.4)}{k^2 + 1}$ for comparison}
    \label{fig:k-sim}
\end{figure}

\begin{figure}[t!]
    \begin{subfigure}{0.49 \textwidth}
        \centering
        \includegraphics[width = \linewidth]{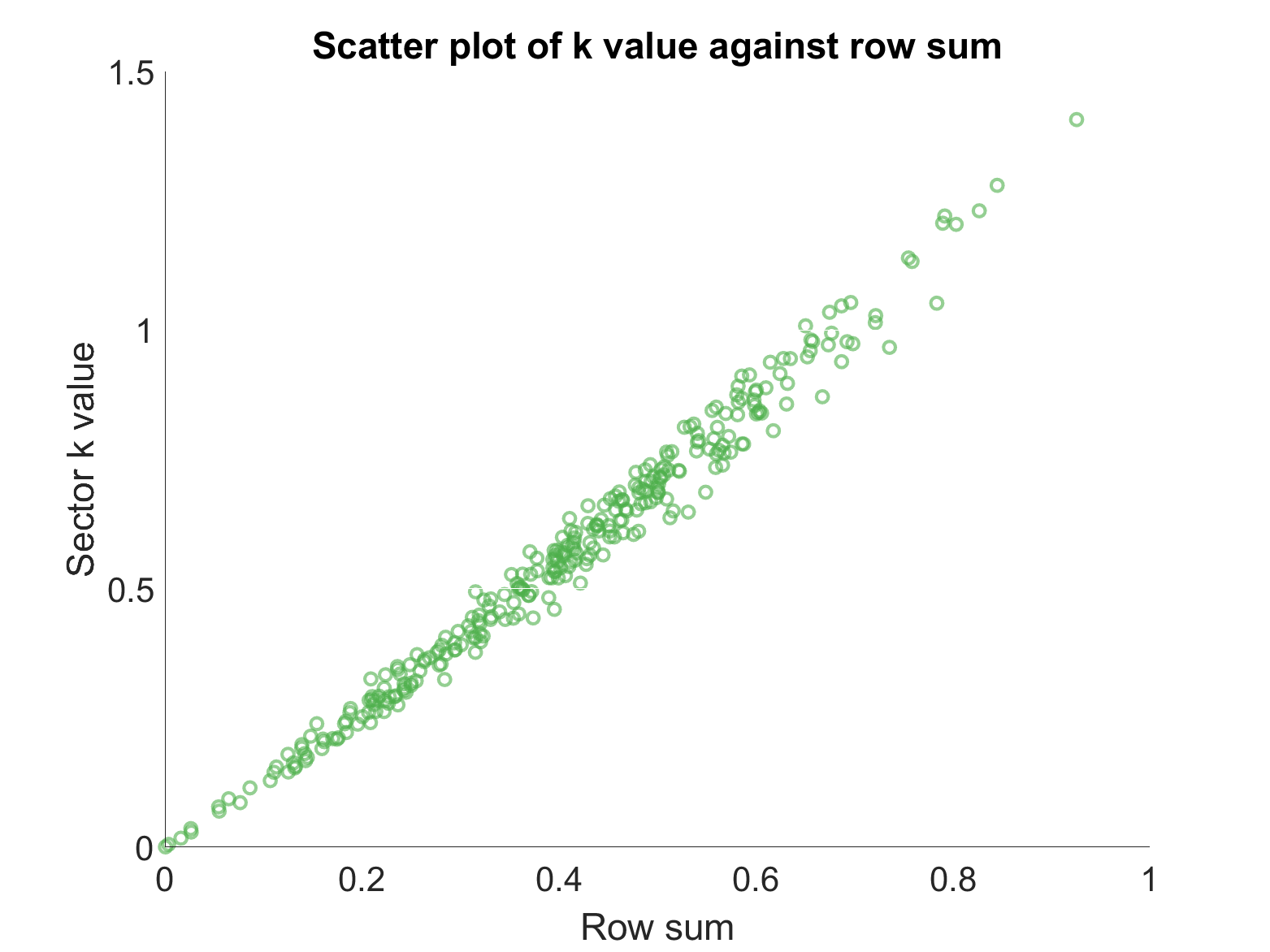}
        \caption{Probability of directed edge $= 0.8$}
        \label{fig:scat-plot-1}
    \end{subfigure}
    \begin{subfigure}{0.49 \textwidth}
        \centering
        \includegraphics[width = \linewidth]{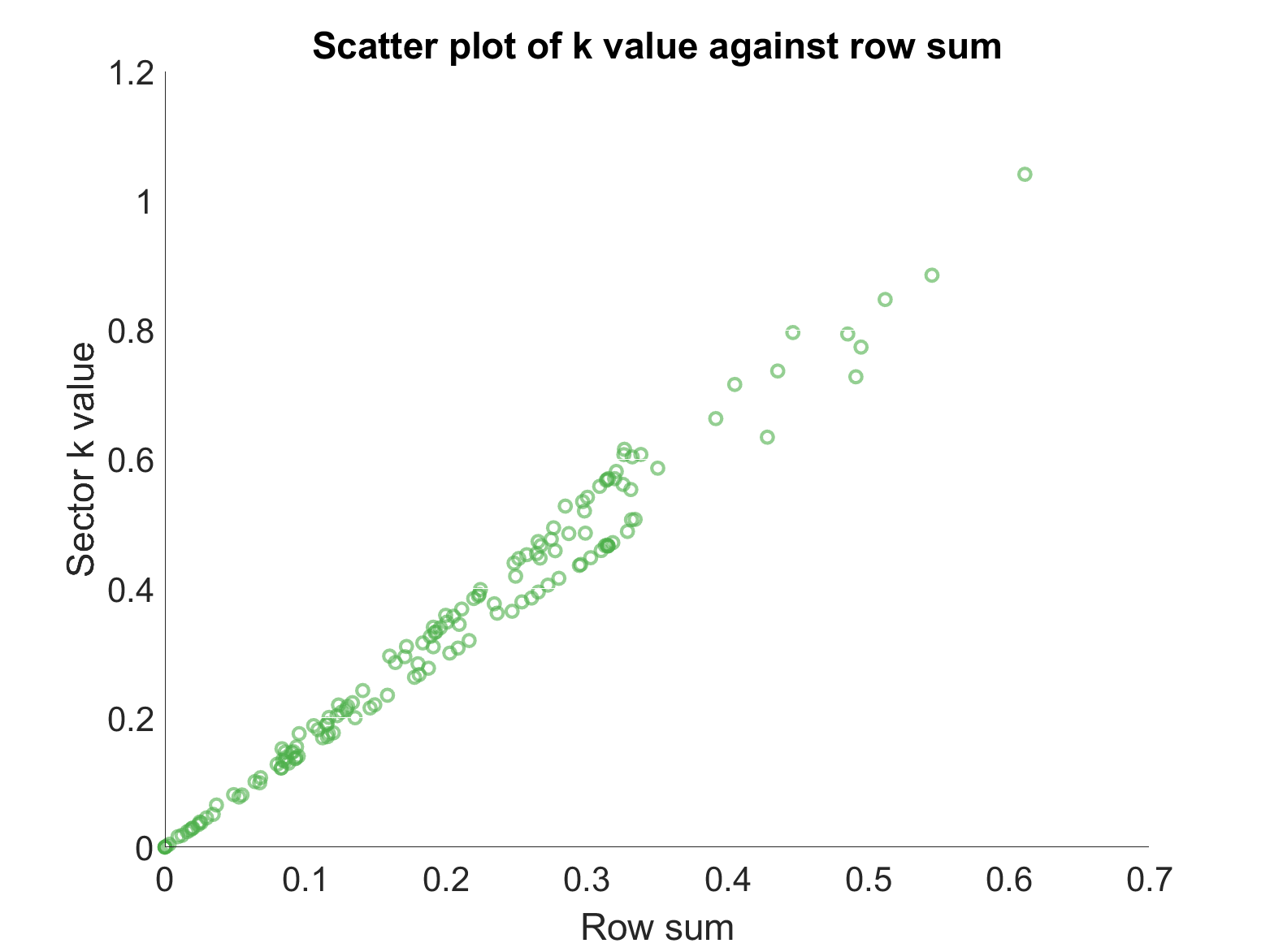}
        \caption{Probability of directed edge $= 0.2$}
        \label{fig:scat-plot-2}
    \end{subfigure}
    \begin{subfigure}{0.49 \textwidth}
        \centering
        \includegraphics[width = \linewidth]{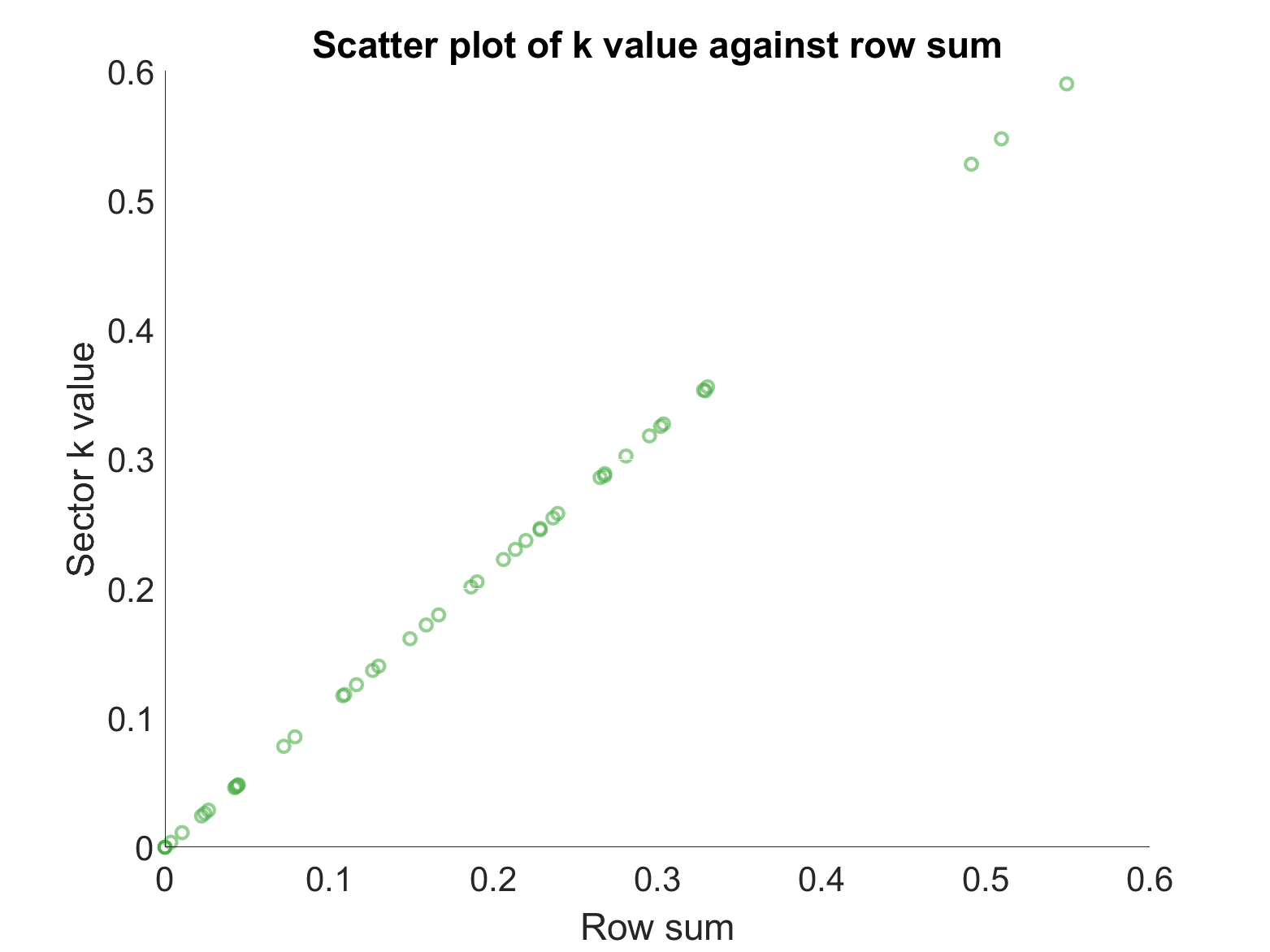}
        \caption{Probability of directed edge $= 0.2$, and length of longest path to sector $= 1$}
        \label{fig:scat-plot-3}
    \end{subfigure}
    \begin{subfigure}{0.49 \textwidth}
        \centering
        \includegraphics[width = \linewidth]{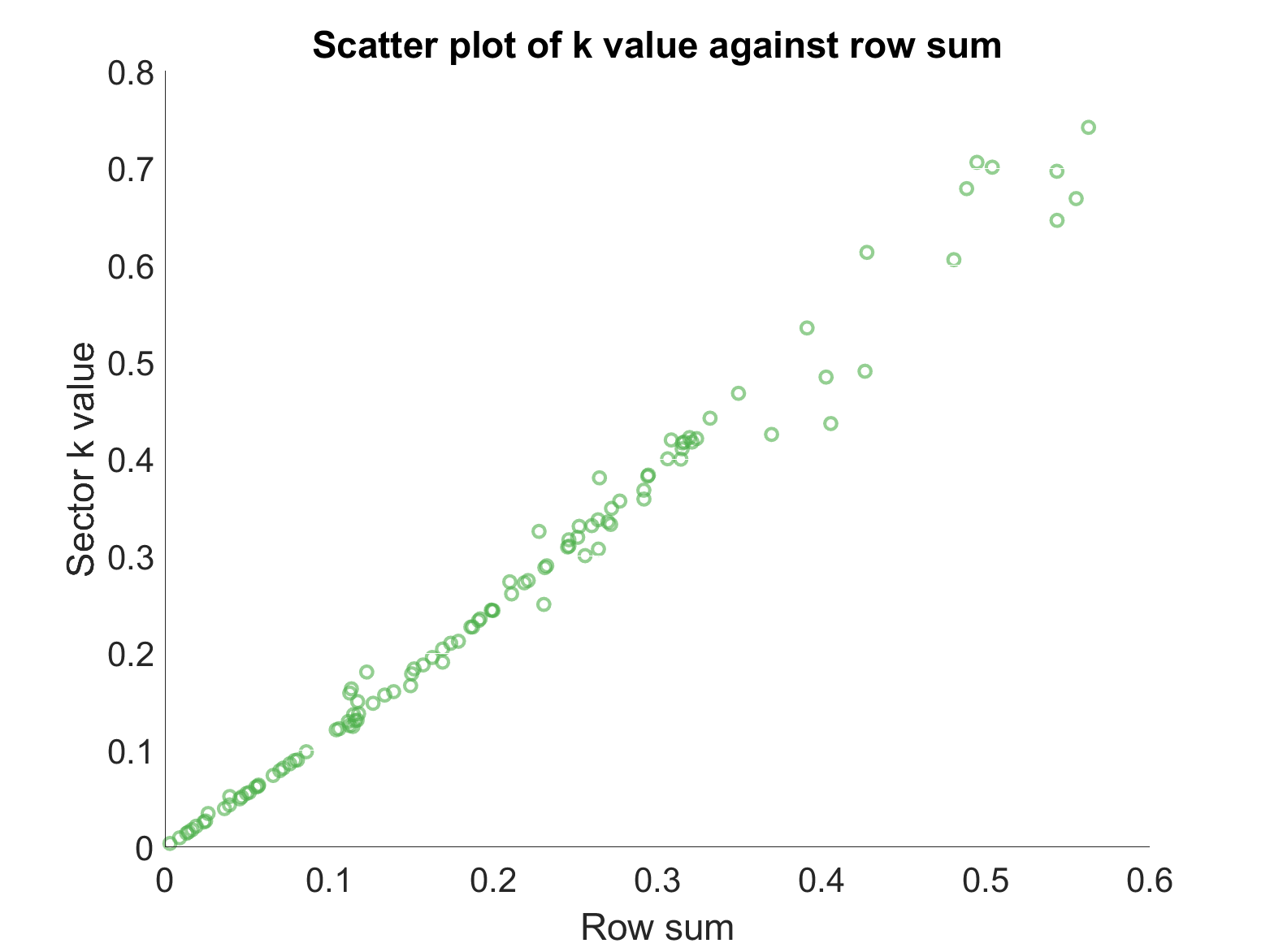}
        \caption{Probability of directed edge $= 0.2$, and length of longest path to sector $> 1$}
        \label{fig:scat-plot-4}
    \end{subfigure}
    \caption{Relationship between $k^*_{\ell}$ and sum of direct spillovers $\sum_{\ell' = 1}^L S_{\ell,\ell'}$}
    \label{fig:scat-plot-14}
\end{figure}

To understand how the value of $k^*_{\ell}$ depends on the matrix $S$ in the case of indirect spillovers, we can return to the definition of the spillover size and $f(k)$. If we assume that $f$ is approximately linear for sectors with indirect spillovers, i.e. $f(k) = f_0 + f_1 k$, then
\begin{equation} \label{eq:comp-spill-relation2}
        k^*_{\ell} = \left(S (f_0 \mathbf{1} + f_1 k^*)\right)_{\ell} \, ,
\end{equation}
where $\mathbf{1}$ is the vector of length $L$ with ones in every entry. Using the identity $(I + f_1 S)^{-1} = \sum_{n = 0}^{\infty} f_1^n S^n$, we can rearrange~\eqref{eq:comp-spill-relation2} 
\begin{equation} \label{eq:k-relation-2}
    k^*_{\ell} = f_0 \sum_{n = 0}^{\infty} f_1^n \left(S^{n + 1} \mathbf{1}\right)_{\ell} \, ,
\end{equation}
which gives a way to estimate the value $k^*_{\ell}$ directly from the initial data. Therefore, combining estimates~\eqref{eq:avg-relation} and~\eqref{eq:k-relation-2}, we can estimate the value of average productivity from the matrix $S$ by
\begin{equation} \label{eq:regression} 
    \int_{\Omega} z m_{\ell}(z)~dz = \bar{z} - \frac{b_0}{\left(f_0 \sum_{\ell'= 1}^L \sum_{n = 0}^{\infty} f_1^n \left(S^{n + 1}\right)_{\ell,\ell'}\right)^{b_1} + b_2} \, .
\end{equation}
The relationship suggests that the average productivity depends on $S^n$ for every $n$ i.e. on indirect spillovers of every path length. Moreover, if $f_1$ is small enough, the effect of a spillover path is decreasing by an order of magnitude for every increase in path length, which agrees with our initial simulations of networks 1--6.

In order to verify the hypothesis, in the final simulations we ran a regression to estimate the parameters $f_0,f_1,b_0,b_1,b_2$ and provide evidence that approximation~\eqref{eq:regression} is accurate. We performed 1000 simulations on networks of ten vertices, with connection probability chosen randomly and uniformly distributed in $[0,1]$, with connection strength chosen randomly and uniformly in $[0,3]$, and with sector sizes $A_{\ell}$ also randomly chosen. We ran a nonlinear regression, of the form~\eqref{eq:k-relation-2}, on sectors with indirect spillovers, to obtain optimal values of $f_0$ and $f_1$. Then, using the optimal values of $f_0$ and $f_1$ we ran a second nonlinear regression, of the form~\eqref{eq:regression}, to find the optimal values of $b_0$, $b_1$ and $b_2$. Table~\ref{tab:regression} gives estimates for the parameters $f_i$ and $b_i$. We found that average productivity does behave approximately according to~\eqref{eq:regression}, with table~\ref{tab:regression} suggesting a statistically significant result. Visually, this can be seen in Figure~\ref{fig:scat-plot-6}, where we plotted~\eqref{eq:regression} using the optimal values of $f_i$ and $b_i$. We also computed estimates for the model
\begin{equation} \label{eq:k-relation-false}
    \int_{\Omega} z m_{\ell}(z)~dz = \bar{z} - \frac{\bar{b}_0}{\left(\bar{f}_0 \sum_{\ell'= 1}^L S_{\ell,\ell'}\right)^{\bar{b}_1} + \bar{b}_2} \, ,
\end{equation}
which assumes average productivity depend on direct spillovers only, and plotted the result in Figure~\ref{fig:scat-plot-6}. Comparing plots~\ref{fig:scat-plot-5} and~\ref{fig:scat-plot-6} shows that the model~\eqref{eq:regression}, which includes the effects of indirect spillovers, provides a more accurate estimate for average productivity than model~\eqref{eq:k-relation-false}, which only accounts for the effect of direct spillovers. This is reconfirmed by the $20\%$ reduction in R--squared error when indirect spillover paths are included in the model. Therefore, indirect spillover paths can not be ignored as a factor determining a sector's productivity.

\section{Conclusion and future research}
We have developed an MFG model of firm--level innovation from a microscopic formulation. The model can be calibrated to fit economic data of spillovers, so its economic validity can be verified. We have been able to prove existence of solutions and, under a smallness assumption on the data, uniqueness. We have investigated numerically how the modelling parameters and the spillover network affects the sector--level productivity, through the development of a simple algorithm that takes advantage of the structure of the proof of existence and uniqueness.

In future work, we hope to compare the MFG model with the socially optimal behaviour, as described by the mean field optimal control problem. We will also use patent--level data to calibrate and test the two models for their accuracy. We hope the comparison between the social optimum and the competitive equilibrium will suggest a method for implementing socially optimal subsidy policies for R\&D.

\begin{table}[b!]
    \centering
    \begin{tabular}{|c|c|c|c|c|c|}
        \hline
       Variable & Coefficient estimate & Standard error & t stat & p value  \\
        \hline
        $f_0$ & 2.29 & $2.13 \times 10^{-3}$ & 1080 &  0 \\
        $f_1$ & $0.483$ & $3.64 \times 10^{-4}$ & 1330 & 0 \\
        $b_0$ & 0.892 & $4.37 \times 10^{-4}$ & 2040 & 0 \\
        $b_1$ & 1.27 & $9.49 \times 10^{-4}$ & 1330 & 0 \\
        $b_2$ & 0.978 & $4.71 \times 10^{-4}$ & 2080 & 0 \\
        \hline
    \end{tabular}
    \caption{Table of regression results related to linear regression~\eqref{eq:regression}}
    \label{tab:regression}
\end{table}
\begin{figure}[b!]
    \begin{subfigure}{0.49 \textwidth}
        \centering
        \includegraphics[width = \linewidth]{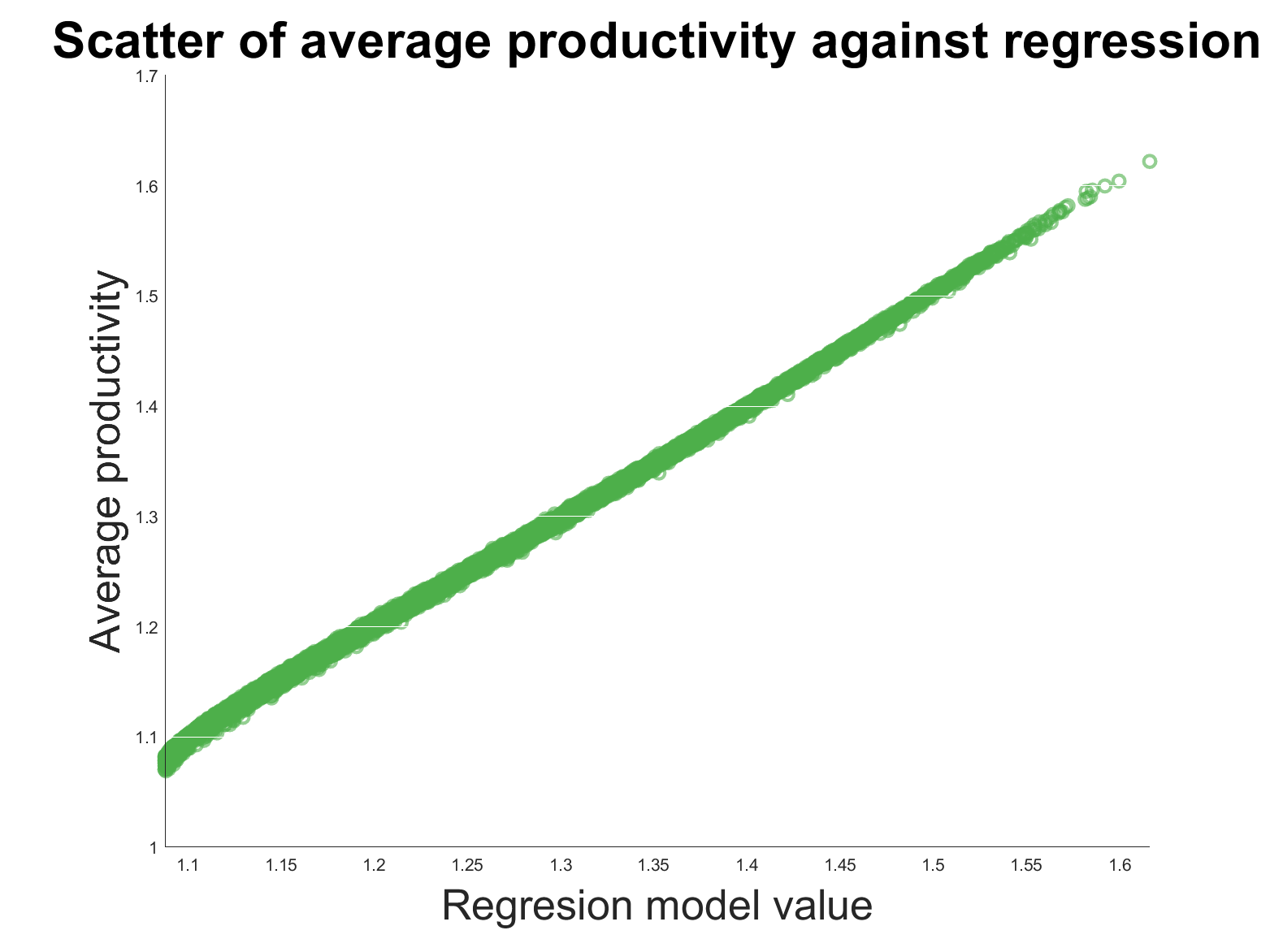}
        \caption{Plot of average productivity against right hand side of~\eqref{eq:regression}, with optimal values for $f_i,b_i$}
        \label{fig:scat-plot-5}
    \end{subfigure}
    \begin{subfigure}{0.49 \textwidth}
        \centering
        \includegraphics[width = \linewidth]{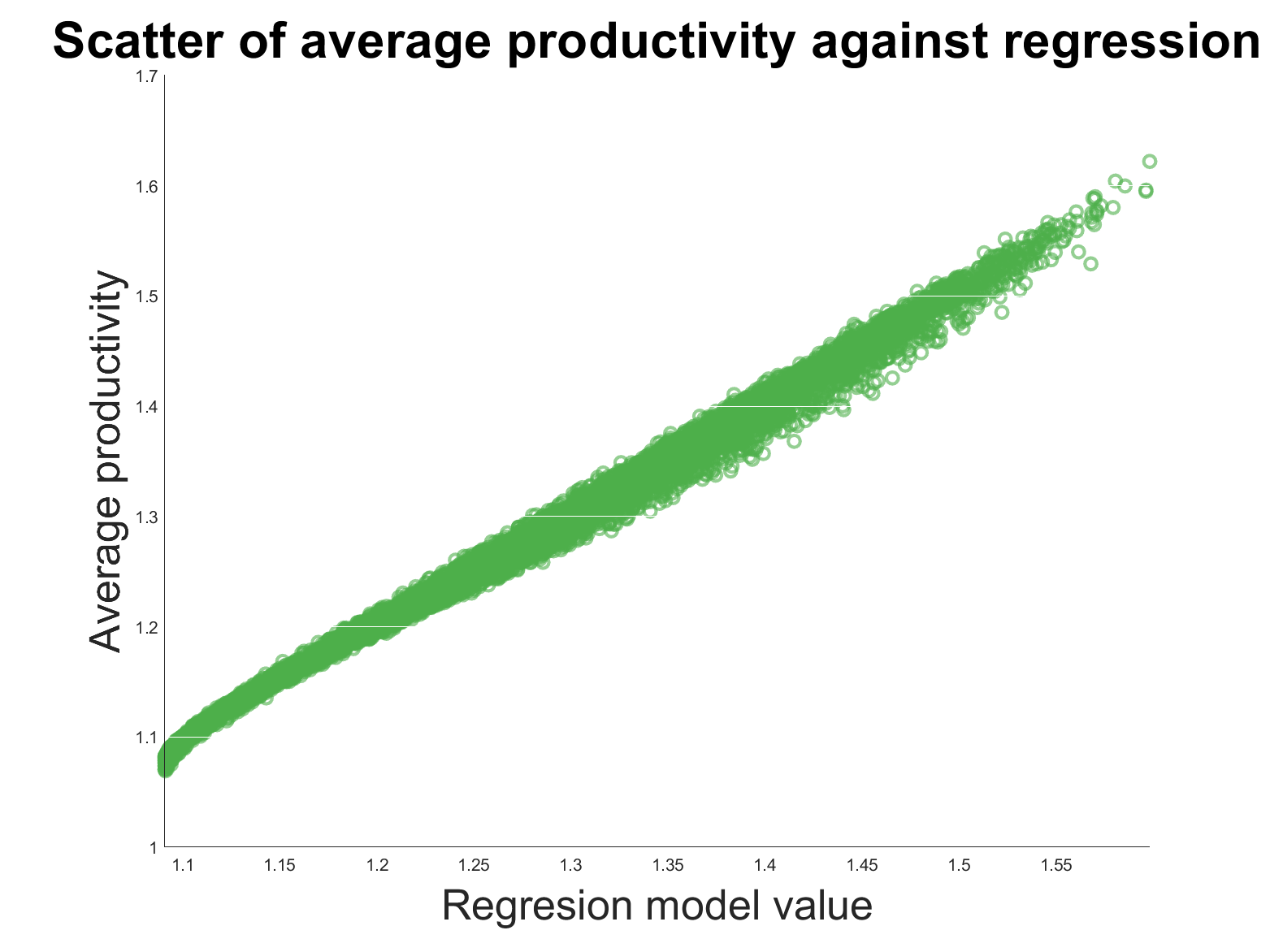}
        \caption{Plot of average productivity against right hand side of~\eqref{eq:k-relation-false}, with optimal values for $\bar{f}_0,\bar{b}_i$}
        \label{fig:scat-plot-6}
    \end{subfigure}
    \caption{Plots of average productivity against~\eqref{eq:regression} with models for $k_{\ell}$ given by considering direct and indirect spillovers~\eqref{eq:k-relation-2} or only direct spillovers~\eqref{eq:k-relation-false}}
    \label{fig:scat-plot-56}
\end{figure}
\appendix
\section{Numerical Methods} \label{sec:appendix}
The numerical method we designed to solve~\eqref{eq:mfg-model} is informed by the structure of the proof of existence and uniqueness. The method of proof relies on the contraction mapping theorem to find a fixed point of the map $\Phi$, defined in Definition~\ref{def:mfg-phi}. We are also required to solve a fixed point problem to find the value of the parameter $B$. In light of this, our numerical method proceeds as follows, after choosing an initial guess $k^0 \in [0,\infty)^L$, $B^0 \in [0,\infty)$ and tolerances $\delta_1,\delta_2$.

\begin{enumerate}
	\item Given $k^i \in [0,\infty)^L$ and $B^i \in [0,\infty)$, solve~\eqref{eq:mfg-alt-HJB},~\eqref{eq:mfg-alt-HJB-bc} using the following method, based on a Newton--Raphson method in a Banach space.
	\begin{enumerate}
		\item Define $F(v) = - \frac{\sigma^2}{2} v'' + \rho v - k v' - (1 - \gamma) \left( \frac{\gamma}{w} \right)^{\frac{\gamma}{1 - \gamma}}(v')^{\frac{1}{1 - \gamma}} - \frac{z^{\alpha}}{B^{\alpha - 1}}$. We want to find zeros of $F(v)$.
		\item We define $dF(v)(u) = - \frac{\sigma^2}{2} u'' + \rho u - k u' - \left(\frac{\gamma}{w} v'\right)^{\frac{\gamma}{1 - \gamma}} u'$, which is the Fr\'echet derivative of $F$.
		\item Denote by $V_0^{k^i_{\ell},B^i}$ the initial guess for the $\ell$th component of the solution to~\eqref{eq:mfg-alt-HJB},~\eqref{eq:mfg-alt-HJB-bc} with $k = k^i_{\ell}$ and $B = B^i$.
		\item Given $V_n^{k^i_{\ell},B^i}$, we compute the next iteration, $V_{n+1}^{k^i_{\ell},B^i}$, using a Newton--Raphson method: $V_{n+1}^{k^i_{\ell},B^i} = V_n^{k^i_{\ell},B^i} - dF\left(V_n^{k^i_{\ell},B^i}\right)^{-1}\left(F\left(V_n^{k^i_{\ell},B^i}\right)\right)$.
		\item Continue iteratively until $\left\|F\left(V_n^{k^i_{\ell},B^i}\right)\right\|_1  \leq \delta_1$ and define $V^{i,\ell} = V_n^{k^i_{\ell},B^i}$
	\end{enumerate}
	\item Given $V^{i}$, compute the solution to~\eqref{eq:mfg-alt-FP} using~\eqref{eq:mfg-alt-FPsol} and denote it by $m^i$
	\item Define $k^{i+1} = \Phi\left(k^i\right)$ and ${B^{i+1} = \left[\sum_{\ell = 1}^L A_{\ell} \int_{\Omega} z^{\alpha} m_{\ell}(z)~dz\right]^{\frac{1}{1 - \alpha}}}$
	\item If $\left\|k^{i+1} - k^i\right\|_1 + |B^{i+1} - B^i| \leq \delta_2$ then stop the iteration process and define the MFG solution $(m,V) = \left(m^i,V^i\right)$. Otherwise return to Step 1.
\end{enumerate}

\bibliographystyle{siamplain}
\bibliography{MFG innovation arxiv}

\end{document}